\documentclass{amsart}

\usepackage{graphicx}
\usepackage{amssymb}
\usepackage{amsmath}
\usepackage{amsfonts}
\usepackage{dsfont}
\usepackage{amsthm}
\usepackage{url}
\DeclareMathOperator*{\var}{var}

\newtheorem{thm}{Theorem}[section]
\newtheorem{prop}[thm]{Proposition}
\theoremstyle{definition}

\theoremstyle{remark}

\theoremstyle{remark}

\theoremstyle{plain}
\newtheorem{lem}[thm]{Lemma}
\theoremstyle{plain}
\newtheorem{Property}[thm]{Property}

\begin{document}

\title[Anderson polymer in fractional environment]{Anderson polymer in a fractional Brownian environment: asymptotic behavior of the partition function}

\author[K. Kalbasi]{Kamran Kalbasi}
\author[T. Mountford]{Thomas S. Mountford}
\author[F. Viens]{Frederi G. Viens}


\begin{abstract}
We consider the Anderson polymer partition function
$$
u(t):=\mathbb{E}^X\Bigl[e^{\int_0^t \mathrm{d}B^{X(s)}_s}\Bigr]\,,
$$
where $\{B^{x}_t\,;\, t\geq0\}_{x\in\mathds{Z}^d}$ is a family of independent fractional Brownian motions all with Hurst parameter $H\in(0,1)$, and $\{X(t)\}_{t\in \mathds{R}^{\geq 0}}$ is a continuous-time simple symmetric random walk on $\mathds{Z}^d$ with jump rate $\kappa$ and started from the origin.  $\mathbb{E}^X$ is the expectation with respect to this random walk.

We prove that when $H\leq 1/2$, the function $u(t)$ almost surely grows asymptotically like $e^{\lambda t}$, where $\lambda>0$ is a deterministic number. More precisely, we show that as $t$ approaches $+\infty$, the expression $\{\frac{1}{t}\log u(t)\}_{t\in \mathds{R}^{>0}}$ converges both almost surely and in the $\mathcal{L}^1$ sense to some positive deterministic number $\lambda$.

For $H>1/2$,  we first show that $\lim_{t\rightarrow \infty} \frac{1}{t}\log u(t)$ exists both almost surely and in the $\mathcal{L}^1$ sense, and equals a strictly positive deterministic number (possibly $+\infty$); hence almost surely $u(t)$ grows asymptotically at least like $e^{\alpha t}$ for some deterministic constant $\alpha>0$. On the other hand, we also show that almost surely and in the $\mathcal{L}^1$ sense, $\limsup_{t\rightarrow \infty} \frac{1}{t\sqrt{\log t}}\log u(t)$ is a deterministic finite real number (possibly zero), hence proving that almost surely $u(t)$ grows asymptotically at most like $e^{\beta t\sqrt{\log t}}$ for some deterministic positive constant $\beta$.

Finally, for $H>1/2$ when $\mathds{Z}^d$ is replaced by a circle endowed with a H\"older continuous covariance function, we show that $\limsup_{t\rightarrow \infty} \frac{1}{t}\log u(t)$ is a deterministic finite positive real number, hence proving that almost surely $u(t)$ grows asymptotically at most like $e^{c t}$ for some deterministic positive constant $c$.

\end{abstract}

\maketitle

\section*{Introduction}

Let $(\Omega^X, \mathcal{F}^X, (\mathcal{F}^X_t)_{t\geq 0}, \mathrm{P}^X)$ be a complete filtered probability space with $\mathrm{P}^X$ being the probability law of a simple symmetric random walk on $\mathds{Z}^d$ indexed by $t\in \mathds{R}^{\geq 0}$ and started at the origin. We denote the jump rate of the random walk by $\kappa$ , the corresponding expectation by $\mathbb{E}^X$ and a random walk sample path by $X(\cdot)$. We also denote by $\mathcal{D}_T$ the space of right-continuous paths $X:[0,T]\rightarrow \mathds{Z}^d$ that have only a finite number of jumps all of which having length one. Let also $\{B^{x}_t\,;\, t\geq0\}_{x\in\mathds{Z}^d}$ be a family of stochastic processes indexed by $\mathds{Z}^d$ and independent of the random walk.\\
The \textit{Anderson polymer model} (with inverse temperature $\beta$) \cite{CranstonComets} is the Gibbs measure on $\mathcal{D}_T$ defined by (if it exists)
\begin{equation}\label{Gibbs measure definition}
\mu^X_{\kappa,\beta,T}(\mathcal{A})=\frac{1}{u(T)}\mathbb{E}^X
\bigl[e^{\beta\int_0^T \mathrm{d}B_s^{X(s)}}\mathbf{1}_{\{X(\cdot)\in\mathcal{A}\}}\bigr] \quad \text{for any event }\mathcal{A}\in\mathcal{D}_T
\end{equation}
where
$$
u(T):=\mathbb{E}^X\Bigl[e^{\beta\int_0^T \mathrm{d}B^{X(s)}_s}\Bigr]\,.
$$
Here the stochastic integral is nothing more than a summation. Indeed, suppose $n$ is the number of jumps of the random walk $X(\cdot)$, the sequence $\{t_i\}_{i=1}^{n}$ are the jump times of $X(\cdot)$ in the time interval $(0,T)$, and define $t_0:=0$ and $t_{n+1}:=T$. Also for each $i$ let $x_i$ be the value of $X(\cdot)$ in the time interval $[t_i,t_{i+1})$. Then we have
$$
\int_0^T \mathrm{d}B^{X(s)}_s=\sum_{i=0}^{n}\bigl(B_{t_{i+1}}^{x_i}-B_{t_{i}}^{x_i}\bigr)\,.
$$
The function $u(\cdot)$ is called the \textit{partition function} of the polymer.

$u(\cdot)$ is also related to the solution of the parabolic Anderson model which is described by the following equation
\begin{equation}\label{general PAM}
    \left\{
        \begin{aligned}
        & \frac{\partial}{\partial t}u(t,x)= \kappa \boldsymbol{\Delta} u(t,x)+ \xi(t,x)\,u(t,x),
\qquad x\in\mathds{Z}^d\;,\; t\geq 0\\
        &u(0,x)=u_o(x)\,,
        \end{aligned}
    \right.
\end{equation}
where $\kappa>0$ is a diffusion constant and $\boldsymbol{\Delta}$ is the discrete Laplacian defined by
${\boldsymbol{\Delta} f(x):= \frac{1}{2d}\sum_{|y-x|=1}\bigl(f(y)-f(x)\bigr)}$. The potential $\{\xi(t,x)\}_{t,x}$ can be a random or deterministic field or even a Schwartz distribution. For more on the parabolic Anderson model we refer to the classical work of Carmona and Molchanov \cite{Carmona}, as well as the surveys  \cite{GaertnerKoenig, cranston, koenig16}.

By \cite{Carmona}, if for every $y\in\mathds{Z}^d$ the function $\{\xi(t,y)\}_{t}$ is locally integrable in $t$, and if the following Feynman-Kac formula is finite, then it actually solves Equation \eqref{general PAM}
\begin{equation}\label{Feynman-Kac}
u(t,x)=\mathbb{E}^X\Bigl[u_o(X(t))e^{ \int_0^t \xi(s, X(t-s))\mathrm{d}s}\Bigr]=\mathbb{E}^X\Bigl[u_o(X(t))e^{ \int_0^t \mathrm{d}B_s^{ X(t-s)}}\Bigr]\,,
\end{equation}
where $B_t^y:=\int_0^t \xi(s,y)\mathrm{d}s$, and $X(t)$ is a simple symmetric random walk on $\mathds{Z}^d$, with jump rate $\kappa$, started at $x \in \mathds{Z}^d$ and independent of the family $\{\xi(t,y)\}_{t,y}$.

The parabolic Anderson model (PAM) has been extensively studied both when the potential $\xi(t,x)$ is a real-valued field (see e.g. \cite{EHM15,GHM12,koenig16} and the references therein) and when it is a white Gaussian noise which is a distributional-valued field (see e.g. \cite{Carmona,BorodinCorwin} and their cited references). On the contrary, very little is known on the PAM driven by distributional-valued potentials other than the white Gaussian noise.
We study the asymptotic behavior of $u(\cdot)$ when the potential $\xi(t,x)$ is a fractional noise by which we mean the corresponding $\{B^{x}_t\,;\, t\geq0\}_{x\in\mathds{Z}^d}$ is a family of independent fractional Brownian motions all with the same Hurst parameter $H\in(0,1)$.
Fractional Brownian motion (fBm) as a generalization of Brownian motion, is widely used to incorporate long-range spatial or temporal correlations. Many phenomena in physics, biology, economy and telecommunications show long range memory (see e.g. \cite{Qian} and references therein).

When the parabolic Anderson model is driven by fractional (or white) noise which is the \textit{formal derivative} of fractional (or standard) Brownian motion, the parabolic Anderson equation is formulated in the following mild sense
\begin{equation}\label{e1}
    \left\{
        \begin{aligned}
        & u(t,x)-u(0,x)=\kappa \int_0^t \boldsymbol{\Delta} u(s,x)\,\mathrm{d}s+\int_0^t u(s,x)\, \mathrm{d}B_s^x\\
        &u(0,x)=u_o(x)
        \end{aligned}
    \right.
\end{equation}
where $\{B^{x}_t\,;\, t\geq0\}_{x\in\mathds{Z}^d}$ is a family of independent fractional Brownian motions all with Hurst parameter $H\in(0,1)$, and the stochastic integral is of the Stratonovich type \cite{Carmona,Kamran}.

It has been shown that the Feynman-Kac formula \eqref{Feynman-Kac} solves Equation \eqref{e1} when it is driven by white noise \cite{Carmona} and when driven by fractional noise of any arbitrary Hurst parameter \cite{Kamran}.

The asymptotic behavior of $u(\cdot)$ has been studied in \cite{Carmona} and \cite{Mountford} for the case of Brownian motion i.e. $H=1/2$. It has been demonstrated \cite{Carmona,Mountford} that almost surely, $\frac{1}{t} \log u(t)$ converges to some deterministic positive constant $\lambda$ which is called the Lyapunov exponent of $u(\cdot)$. These proofs make use of subadditivity properties and independent increments of the Brownian motion which no longer apply to the general case of $H\in(0,1)$.

In \cite{Viens} the parabolic Anderson model on a circle ($\mathcal{S}^1$) with Riemann-Liouville fractional Brownian environment was considered. For $H\leq 1/2$, assuming quite strong conditions on $H$, $\kappa$, and the spatial covariance, it proves that $\{\frac{1}{n}\log u(n)\}_{n\in \mathds{N}}$ converges to some deterministic positive number. For $H>1/2$, it tries to show that $\log u(t)$ grows asymptotically faster than $\frac{t^{2H}}{\log t}$, which is in contrast with our results in Section \ref{Compact Case} where we show that in the compact-space setup (for example a circle), $\log u(t)$ grows linearly for $H>1/2$ as well.

In this paper we consider
\begin{equation}\label{u}
u(t):=\mathbb{E}^X\Bigl[e^{\int_0^t \mathrm{d}B^{X(s)}_s}\Bigr]\,,
\end{equation}
where $\{B^{x}_t\,;\, t\geq0\}_{x\in\mathds{Z}^d}$ is a family of independent fractional Brownian motions all with Hurst parameter $H\in(0,1)$. It should be noted that $\beta$ in Equation \eqref{Gibbs measure definition} plays no role in our arguments, hence for the sake of simplicity we take $\beta$ equal to $1$. Let also
\begin{equation}
U(t):=\mathbb{E} \log u(t)\,,
\end{equation}
where ``$\mathbb{E}$'' denotes expectation with respect to the fractional Brownian motion field.\\
Although we assume that the fractional Brownian motions associated to different sites of $\mathds{Z}^d$ are independent, our results remain valid for much more general spatial covariance structures.

We summarize the results of this paper by the following two theorems.
\begin{thm} With the above notations, we have:\\
\textbf{I. } For $H\leq 1/2$, the limit of $\{\frac{1}{t}\log u(t)\}_{t\in \mathds{R}^{>0}}$, as $t$ approaches $\infty$, exists both almost surely and in the $\mathcal{L}^1$ sense, and equals some strictly positive finite deterministic number.\\
\textbf{II. }For $H>1/2$, the limit of $\{\frac{1}{t}\log u(t)\}_{t\in \mathds{R}^{>0}}$, as $t$ approaches $\infty$, exists both almost surely and in the $\mathcal{L}^1$ sense, and equals some strictly positive deterministic number (possibly $+\infty$). Moreover, $\limsup_{t\rightarrow \infty} \frac{u(t)}{t\sqrt{\log t}}$ is a finite deterministic number.
\end{thm}

As a by-product we also prove the next theorem that provides a time-linear upper bound on the Anderson polymer partition function in the set-up of \cite{Viens}. Setting the stage, let $\{B^{x}_t\,;\, t\geq0\}_{x\in\mathds{R}}$ be a family of fractional Brownian motions of Hurst parameter $H\in(0,1)$ with the following covariance structure
$$
\mathbb{E}\bigl(B_t^x B_s^y\bigr)=R_H(t,s) \, Q(x,y)\,,
$$
where $Q:\mathds{R}\times\mathds{R}\rightarrow \mathds{R}$ is a positive semi-definite function $2\pi$-periodic in both coordinates, and $R_H(t,s)$ is the fBm covariance function, i.e.
$$
R_H(t,s):=\frac{1}{2}(|t|^{2H}+|s|^{2H}-|t-s|^{2H})\,.
$$
Let $u_c(\cdot)$ be defined as follows
\begin{equation}\label{u-c}
u_c(t):=\mathbb{E}^X\Bigl[e^{\int_0^t \mathrm{d}B^{X(s)}_s}\Bigr]\,,
\end{equation}
where  $X(\cdot)$ is a symmetric random walk on $\mathds{R}$ with unit jumps and started from the origin, or equivalently a simple symmetric random walk on $\mathds{Z}$ started from zero.\\
Note that this set-up is equivalent to the Anderson model over $\mathcal{S}^1$ (unit circle). Then we have
\begin{thm}
Suppose $Q(\cdot,\cdot)$ is H\"older continuous of order $\alpha>0$, in the sense that there exist positive constants $C$ and $\alpha$ such that
$$
|Q(x,y)-\frac{1}{2}Q(x,x)-\frac{1}{2}Q(y,y)|\leq C |x-y|^\alpha \quad \text{for every }x,y\in \mathds{R}\,.
$$
Then there exists a deterministic constant $0<\lambda<\infty$ such that almost surely $u_c(t)\leq e^{\lambda t}$ for $t$ sufficiently large.
\end{thm}

The organization of this paper is as follows:

In Section \ref{preliminaries} we gather some background material which will be used in the succeeding sections.

In Section \ref{Approximation}, we show that the main contribution to $U(t)$ comes from those random walk occurrences that have at most $\mathfrak{N}_t$ (to be defined there) number of jumps over the time interval $[0,t]$. We denote by $\widehat{U}(t)$ the part of $U(t)$ that comes from this kind of random walk occurrences.
We also show that as far as the asymptotic behavior of $\{\widehat{U}(t)\}_{t \in \mathds{R}^{>0}}$ is concerned we can confine our attention to the integer values only, i.e. when $t\in \mathds{N}$.

In Section \ref{Lipschitz continuity of residues of fBm increments}, we develop a Lipschitz inequality that will serve as a building block for all our subsequent arguments.

In Section \ref{Super-additivity}, we prove an approximate super-additivity for $\widehat{U}(\cdot)$. This would then imply the convergence of $\frac{1}{t}\widehat{U}(t)$ as $t$ goes to infinity.

Section \ref{Quenched limits} is devoted to the quenched asymptotic behavior. In mathematical physics terminology the quenched statements are those statements that are formulated almost surely. Here we seek the almost sure behavior of $\log u(t)$ when $t$ approaches infinity. In this section we show that $\log u(\cdot)$ has the same asymptotic behavior as $\widehat{U}(\cdot)$. In particular we obtain limits over the positive real $t$'s instead of just positive integers.

In Section \ref{Lower Bound}, we establish a strictly positive asymptotic lower bound on $\{\frac{1}{t}\widehat{U}(t)\}_t$, for any $\kappa$ and $H\in(0,1)$. Hence along with the super-additivity result, it shows that $\widehat{U}(t)$ grows in $t$ at least as fast as $\lambda t$ for some strictly positive $\lambda$.

Section \ref{Upper Bounds} deals with finding an asymptotic upper bound on $\{\frac{1}{t}\widehat{U}(t)\}_t$. There, a finite asymptotic upper bound is easily found for the case of $H\leq 1/2$.  For $H>1/2$, we show that $\{\frac{1}{t}\widehat{U}(t)\}_t$ is asymptotically bounded by $Ct\sqrt{\log t}$ for some positive constant $C$.

Finally in Section \ref{Compact Case}, we deal with the compact-space setup of \cite{Viens}. We show that compactness can be utilized to improve our $t\sqrt{\log t}$ upper bound to a linear one.
\section{Preliminaries}\label{preliminaries}
We recall some basic elements that will be used in later sections.

\subsection{Fractional Brownian Motion}
A Gaussian random process $\{B_t\}_{t\in \mathds{R}}$ is called a fractional Brownian motion (fBm) of Hurst parameter $H\in(0,1)$ if it has continuous sample paths and its covariance function is of the following form:
\begin{equation}\label{fractional covariance function}
\mathbb{E}\bigl(B_t B_s\bigr)=R_H(t,s):=\frac{1}{2}(|t|^{2H}+|s|^{2H}-|t-s|^{2H}).
\end{equation}
\begin{Property}\label{positive or negative correlation} For a fractional Brownian motion with Hurst parameter $H$, the increments over disjoint intervals are positively correlated for $H>1/2$, negatively correlated for $H<1/2$, and independent for $H=1/2$ (see e.g. \cite{Mishura}). In other words, for a fractional Brownian motion $\{B_t\}_t$ of Hurst parameter $H$, and for any $s<t\leq u < v$ the covariance $\mathbb{E}\bigl((B_{t}- B_{s})(B_{v}- B_{u})\bigr)$ is positive for $H>1/2$, negative for $H<1/2$, and zero for $H=1/2$.
\end{Property}
\begin{Property}
A fractional Brownian motion $\{B_t\}_t$, of Hurst parameter $H \in (0,1)$, can be represented as a Volterra process \cite{NualartStochasticIntegration}
\begin{equation}\label{Volterra representation}
B_t=\int_0^t K_H(t,s) \mathrm{d}W_s\,,
\end{equation}
where $W_s$ is a standard Brownian motion and the stochastic integral is in the It\=o sense. $K_H(t,s)$ is the square-integrable kernel defined for every $0<s<t$ as follows\\
For $H > 1/2$:
$$
K_H(t,s):=c_H \int_s^t (u-s)^{H-\frac{3}{2}}\,(\frac{u}{s})^{H-\frac{1}{2}} \mathrm{d}u\,,
$$
and for $H \leq 1/2$:
$$
K_H(t,s):=c'_H \Bigl(\bigl(\frac{t}{s}\bigr)^{H-\frac{1}{2}}(t-s)^{H-\frac{1}{2}}
-(H-\frac{1}{2})s^{\frac{1}{2}-H}\int_s^t u^{H-\frac{3}{2}}(u-s)^{H-\frac{1}{2}} \mathrm{d}u\Bigl)\,,
$$
where $c_H$ and $c'_H$ are positive constants that depend only on $H$.\\

For $H < 1/2$, we have the following equality
$$
K_H(t,s)= \lim_{\alpha \downarrow s} \int_{\alpha}^t \frac{\partial K_H}{\partial t}(u,s) \mathrm{d}u + c'_H \bigl(\frac{\alpha}{s}\bigr)^{H-\frac{1}{2}}(\alpha-s)^{H-\frac{1}{2}}.
$$
This shows that for any $H\in(0,1)$ and any $0<s<t_1<t_2$ we have
$$
\begin{aligned}
K_H(t_2,s)-K_H(t_1,s)
&=\int_{t_1}^{t_2} \frac{\partial K_H}{\partial t}(u,s) \mathrm{d}u\\
&=c_H \int_{t_1}^{t_2} (u-s)^{H-\frac{3}{2}}\,(\frac{u}{s})^{H-\frac{1}{2}} \mathrm{d}u
\,,
\end{aligned}
$$
where $c_H:=c'_H (H-\frac{1}{2})$.
\end{Property}

\subsection{Malliavin Calculus}
Let $(\Omega,\mathcal{F},P)$ be a probability space and $\mathbf{G}$ a Gaussian linear space on it. Let also $\mathbf{H}$ be a Hilbert space with the isometry $\textbf{W}: \mathbf{H} \rightarrow \mathbf{G}$.
Define $\mathcal{S}$ as the space of random variables $F$ of the form:
$$
F=f\bigl(\textbf{W}(\varphi_1) , \dots , \textbf{W}(\varphi_n)\bigr)\,,
$$
where $\varphi_i \in \mathbf{H} , f\in C^\infty(\mathds{R}^n)$, $f$ and all its partial derivatives have polynomial growth. The Malliavin derivative of $F$, $\nabla F$, is defined (see e.g. \cite{Janson,Ustunel}) as an $\mathbf{H}$-valued random variable given by
$$
\nabla F:= \sum _ {i=1}^{n} \frac{\partial f}{\partial x_i}(\textbf{W}(\varphi_1)  , ... , \textbf{W}(\varphi_n) ). \varphi_i
$$
The operator $\nabla$ is closable from $L^2(\Omega)$ into $L^2(\Omega ;  \mathbf{H})$ and one defines the Sobolev space $\mathbb{D}^{1,2}$ as the closure of $\mathcal{S}$ with respect to the following norm \cite{Janson}:
$$
\|F\|_{1,2}= \sqrt{\mathbb{E}(F^2)+\mathbb{E}(\|\nabla F\|^2_\mathbf{H})}.
$$
For more on Malliavin calculus we refer to \cite{Janson}.

Let $\{B_t^x\,;\, t\in\mathds{R}\}_{x\in \mathds{Z}^d}$ be a family of independent fractional Brownian motions indexed by $x\in \mathds{Z}^d$ all with Hurst parameter $H$.\\
Let $\mathcal{H}$ be the Hilbert space defined by the completion of the linear span of indicator functions $\mathbf{1}_{[0,t]\times\{x\}}$ for $t \in \mathds{R}$ and $x\in \mathds{Z}^d$ under the scalar product
$$
\langle\mathbf{1}_{[0,t]\times\{x\}} , \mathbf{1}_{[0,s]\times\{y\}}\rangle_\mathcal{H}=R_H(t,s)\,\delta_x(y)\,,
$$
where $\delta$ is the Kronecker delta, and $R_H$ is as in \eqref{fractional covariance function}. For negative $t$ we assume the convention ${\mathbf{1}_{[0,t]\times\{x\}} := -\mathbf{1}_{[t,0]\times\{x\}}}$.\\
The mapping $\textbf{B}(\mathbf{1}_{[0,t]\times\{x\}}):= B_t^x$ can be extended to a linear isometry from $\mathcal{H}$ onto the Gaussian space spanned by $\{B_t^x)\,;\, t\in\mathds{R}, x\in \mathds{Z}^d\}$. This is the only setting to which we will apply Malliavin calculus in this paper.

\subsection{Concentration inequalities}

Let $(\Omega, \mathcal{F}, P)$ be a probability space, $\mathcal{H}$ be a Gaussian Hilbert space on it and $\mathcal{F}(\mathcal{H})$ be the sigma algebra generated by $\mathcal{H}$ \cite{Janson}. The following theorem shows that the probability distribution of a Malliavin derivable random variable with bounded derivative decays exponentially away from its mean value. We will use this theorem in Section \ref{Quenched limits} for establishing the quenched limits.
\begin{thm}[B.8.1 in \cite{Ustunel}]\label{Ustunel's thm}
Suppose that $\varphi \in \mathbb{D}^{1,p}$ for some $p>1$ with $\nabla \phi \in \mathcal{L}^\infty(\Omega ; \mathcal{H})$, i. e. $\|\nabla \varphi\|_\mathcal{H}$ is almost surely bounded. Then we have the following tail probability estimate:
\begin{equation}
P\{\omega\; ; \; |\varphi(\omega)-\mathbb{E}[\varphi]|>c\}\leq 2 \, exp \{\frac{-\,c^2}{\;2\,\|\nabla \varphi\|_{\mathcal{L}^\infty(\Omega ; \mathcal{H})}^2}\}
\end{equation}
\end{thm}

Dudley's theorem or Dudley's entropy bound \cite{Lifshits} provides a tight upper bound on the expectation of the maximum of a family of Gaussian random variables.
\begin{thm}[\bfseries Dudley]\label{Dudley theorem}
Let $\{X_t\}_{t\in T}$ be a separable centered Gaussian process indexed by some topological space $T$ and $\rho$ be the pseudo-metric on $T$ defined by $\rho(s,t):=\sqrt{\mathbb{E}(X_t-X_s)^2}$. Then we have
\begin{equation}\label{Dudley}
\mathbb{E}(\sup_{t\in T}\,X_t) \leq K\int_0^\infty \sqrt{\log N(\varepsilon)}\, \mathrm{d}\varepsilon\,,
\end{equation}
where $N(\varepsilon)$ is the minimum number of $\rho$-balls of radius $\varepsilon$ required to cover $T$, and $K$ is a universal positive constant.
\end{thm}

Borell's inequality \cite{ConcentrationMeasureLedoux} shows that under very weak conditions, the maximum of a family of Gaussian random variables concentrates around its mean, and away from its mean its probability tails decay exponentially.
\begin{thm}[\bfseries Borell's inequality]\label{Borell}
Let $T$ be a topological space and $\{X_t\}_{t\in T}$ be a separable centered Gaussian process with $\sup_{t\in T}\,X_t\;<\; \infty$ almost surely. Then \cite{ConcentrationMeasureLedoux} the expectation
$\mathbb{E}(\sup_{t\in T}\,X_t)$ is finite and for any $c>0$
$$
\mathbb{P}\Big( |\sup_{t\in T}\,X_t-\mathbb{E}(\sup_{t\in T}\,X_t)|\geq \lambda \Bigl)\leq 2e^{-\frac{\lambda^2}{2\sigma_T^2}}\,,
$$
where $\sigma_T^2:=\sup_{t\in T} \mathbb{E}(X_t^2)$.
\end{thm}
\section{Constraining the Number of Jumps and Quantization}\label{Approximation}
In this section we show that as far as the asymptotic behavior is concerned we can confine our attention to only those random walk occurrences that have a specified maximum number of jumps.

For any positive real number $t$ we define $\mathfrak{N}_t$ as follows
\begin{equation}\label{N_t}
\mathfrak{N}_t:=
\begin{cases}
\lfloor t^2\rfloor & \quad \text{for } H>1/2\\
\lfloor \rho \kappa t\rfloor & \quad \text{for } H\leq1/2\,,
\end{cases}
\end{equation}
where $\rho:=\max\{e^6, \kappa^{-1}\}$, and $\lfloor x \rfloor$ denotes the floor of $x$, i.e. the largest integer not greater than $x$.

For $t>0$, let $\mathcal{A}_t$ be the event that the number of jumps of the random walk in the time interval $[0,t]$ is less than or equal to $\mathfrak{N}_t$, and define $\widehat{U}(t)$ as follows
$$
\widehat{U}(t):=\mathbb{E} \log \mathbb{E}^X\bigl[e^{ \int_0^t \mathrm{d}B^{X(s)}_s} \mathbf{1}_{\mathcal{A}_t}\bigr]\,.
$$

We have the following proposition:

\begin{prop}
For any function $f:\mathds{R}^{>0}\rightarrow \mathds{R}^{>0}$ which satisfies $\alpha|s-t|\leq|f(s)-f(t)|\leq \beta|s-t|^p$ for some fixed positive numbers $\alpha$, $\beta$, and $p$, we have
$$
\limsup_{t\rightarrow \infty} \frac{U(t)}{f(t)}
=\limsup_{\substack{n\rightarrow \infty \\ n \in \mathds{N}}}\frac{\widehat{U}(n)}{f(n)}\,,
$$
and
$$
\liminf_{t\rightarrow \infty} \frac{U(t)}{f(t)}
=\liminf_{\substack{n\rightarrow \infty \\ n \in \mathds{N}}}\frac{\widehat{U}(n)}{f(n)}\,.
$$
\end{prop}
\begin{proof}
We show in the first part that $\frac{U}{f}$ and $\frac{\widehat{U}}{f}$ are asymptotically very close, and then in the second part we show that the asymptotic behavior of $\frac{\widehat{U}}{f}$ over integers is the same as over real numbers.

\textbf{First part: } We should examine how close $U(t)$ is to $\widehat{U}(t)$. Define $\mathcal{S}_X^t:=\int_0^t \mathrm{d}B^{X(s)}_s$. Using the inequality $log(1+a)\leq a$ and then Cauchy-Schwarz we have
$$
\begin{aligned}
U(t)-\widehat{U}(t)&=\mathbb{E} \log \Bigl(1 + \frac{\mathbb{E}^X \bigl[ e^{\mathcal{S}_X^t} \mathbf{1}_{\mathcal{A}_t^c}\bigr]}{\mathbb{E}^X \bigl[ e^{\mathcal{S}_X^t} \mathbf{1}_{\mathcal{A}_t}\bigr]}\Bigr)\\
&\leq \mathbb{E}\Bigl(\frac{\mathbb{E}^X \bigl[ e^{\mathcal{S}_X^t} \mathbf{1}_{\mathcal{A}_t^c}\bigr]}{\mathbb{E}^X \bigl[ e^{\mathcal{S}_X^t} \mathbf{1}_{\mathcal{A}_t}\bigr]}\Bigr)\\
&\leq \sqrt{\mathbb{E}\Bigl(\mathbb{E}^X \bigl[ e^{\mathcal{S}_X^t} \mathbf{1}_{\mathcal{A}_t^c}\bigr]\Bigr)^2}\sqrt{\mathbb{E}\Bigl(\mathbb{E}^X \bigl[ e^{\mathcal{S}_X^t} \mathbf{1}_{\mathcal{A}_t}\bigr]\Bigr)^{-2}}\,,
\end{aligned}
$$
where $\mathcal{A}_t^c$ is the complement of $\mathcal{A}_t$.\\
As $x^{-2}$ is convex and $\mathcal{S}_X^t$ is Gaussian , we have
$$
\mathbb{E}\Bigl(\mathbb{E}^X \bigl[ e^{\mathcal{S}_X^t} \mathbf{1}_{\mathcal{A}_t}\bigr]\Bigr)^{-2}
\leq p_{_{\mathcal{A}_t}} ^{-3} \mathbb{E}\,\mathbb{E}^X \bigl[e^{-2\mathcal{S}_X^t}  \mathbf{1}_{\mathcal{A}_t}\bigr]
\leq p_{_{\mathcal{A}_t}} ^{-3} \mathbb{E}^X \bigl[e^{2 var(\mathcal{S}_X^t)}  \mathbf{1}_{\mathcal{A}_t}\bigr]\,,
$$
where $p_{_{\mathcal{A}_t}}$ is the probability of $\mathcal{A}$, and $var(\cdot)$ denotes variance with respect to the environment, i.e. the fractional Brownian motion field.\\
For the other term, again by Cauchy-Schwarz we have
$$
\mathbb{E}\Bigl(\mathbb{E}^X \bigl[ e^{\mathcal{S}_X^t} \mathbf{1}_{\mathcal{A}_t^c}\bigr]\Bigr)^2 \leq p_{_{\mathcal{A}_t^c}} \,\mathbb{E}\,\mathbb{E}^X \bigl[ e^{2\mathcal{S}_X^t} \mathbf{1}_{\mathcal{A}_t^c}\bigr]
\leq p_{_{\mathcal{A}_t^c}} \, \mathbb{E}^X \bigl[e^{2 var(\mathcal{S}_X^t)}  \mathbf{1}_{\mathcal{A}_t^c}\bigr]
\,,
$$
where $p_{_{\mathcal{A}_t^c}}$ is the probability of $\mathcal{A}_t^c$. So we have
\begin{equation}\label{Difference of U(t)-Û(t)}
0 \leq U(t)-\widehat{U}(t) \leq \, p_{_{\mathcal{A}_t}} ^{-3/2} \, p_{_{\mathcal{A}_t^c}}^{1/2} \, \mathbb{E}^X \bigl[e^{2 var(\mathcal{S}_X^t)}  \mathbf{1}_{\mathcal{A}_t^c}\bigr]  \, \mathbb{E}^X \bigl[e^{2 var(\mathcal{S}_X^t)}  \mathbf{1}_{\mathcal{A}_t}\bigr]
\end{equation}

i) For $H>1/2$: In this case as the increments of the fBm are positively correlated (property \ref{positive or negative correlation}), the maximum variance is achieved when the random walk stays on a single site and never moves away. So we have
$$
var(\mathcal{S}_X^t)\leq t^{2H}\,.
$$
So
$$
\mathbb{E}^X \bigl[e^{2 var(\mathcal{S}_X^t)}  \mathbf{1}_{\mathcal{A}_t}\bigr] \leq
p_{_{\mathcal{A}_t}} e^{2t^{2H}} \qquad \text{and} \qquad
\mathbb{E}^X \bigl[e^{2 var(\mathcal{S}_X^t)}  \mathbf{1}_{\mathcal{A}_t^c}\bigr] \leq
p_{_{\mathcal{A}_t^c}} e^{2t^{2H}}\,
$$
hence by \eqref{Difference of U(t)-Û(t)}, we have
$$
U(t)-\widehat{U}(t)\leq p_{_{\mathcal{A}_t}} ^{-1} p_{_{\mathcal{A}_t^c}} e^{2t^{2H}}.
$$
Let $N$ denote the number of jumps of the random walk $X(\cdot)$ in the time interval $[0,t]$. Evidently $N$ has a Poisson distribution with mean $\kappa t$. But for a general Poisson random variable $N$ with mean $\lambda$ we have the following tail probability bound \cite{Mitzenmacher}
\begin{equation}\label{poisson tail probability bound}
P(N\geq n)\leq e^{-\lambda}(\frac{e \lambda}{n})^n \qquad \text{for } n> \lambda\,.
\end{equation}
Using this bound, for $t\geq \kappa e^2$ we have
$$
p_{_{\mathcal{A}_t^c}} \leq e^{-\kappa t} (\frac{e \kappa t}{t^2})^{t^2}
\leq e^{-\kappa t} e^{- t^2},
$$
which implies $p_{_{\mathcal{A}_t}} \geq 1/2$. Hence
\begin{equation}\label{distance of U and U hat for H greater than half}
0\leq U(t)-\widehat{U}(t) \leq 2 e^{- t^2}e^{2t^{2H}}\sim \mathcal{O}(e^{-t}).
\end{equation}

ii) For $H\leq 1/2$: Let $N$ be the number of jumps of the random walk $X(\cdot)$ in the time interval $[0,t]$, and let $\{t_i\}_{i=1}^N$ be its jump times. We moreover define $t_0:=0$ and $t_{N+1}:=t$. As the increments of the fBm are negatively correlated in this case (property \ref{positive or negative correlation}), the maximum variance is achieved if the random walk never visits any site more than once. Hence we have
$$
var(\mathcal{S}_X^t)\leq \sum_{i=0}^{N} (t_{i+1}-t_i)^{2H}\leq (N+1) (\frac{t}{N+1})^{2H},
$$
where we have used the concavity of the function $x^{2H}$ and the H\"older's inequality.
So
$$
\begin{aligned}
\mathbb{E}^X \bigl[e^{2 var(\mathcal{S}_X^t)}  \mathbf{1}_{\mathcal{A}_t}\bigr] &\leq
\mathbb{E}^X \bigl[e^{2 (N+1)^{1-2H} t^{2H}}  \mathbf{1}_{\mathcal{A}_t}\bigr]
\leq \mathbb{E}^X \bigl[e^{2 (\rho \kappa t)^{1-2H} t^{2H}}  \mathbf{1}_{\mathcal{A}_t}\bigr]\\
&  \leq e^{2 (\rho \kappa)^{1-2H} t} p_{_{\mathcal{A}_t}}\,,
\end{aligned}
$$
and
$$
\begin{aligned}
\mathbb{E}^X \bigl[e^{2 var(\mathcal{S}_X^t)}  \mathbf{1}_{\mathcal{A}_t^c}\bigr] &\leq
\mathbb{E}^X \bigl[e^{2 (N+1) (\frac{t}{N+1})^{2H}}  \mathbf{1}_{\mathcal{A}_t^c}\bigr]
 \leq \mathbb{E}^X \bigl[e^{2 (N+1) (\rho \kappa)^{-2H}}  \mathbf{1}_{\mathcal{A}_t^c}\bigr] \\
&\leq \mathbb{E}^X \bigl[e^{2 (N+1)}  \mathbf{1}_{\mathcal{A}_t^c}\bigr]
=e^{-\kappa t}\sum_{n\geq \lfloor \rho \kappa t\rfloor} \frac{(\kappa t)^n}{n!} e^{2n}\leq e^{-\kappa t}e^{e^2\kappa t},
\end{aligned}
$$
where we have used the fact that $\rho \kappa \geq 1$.\\
Finally using $\rho \geq e^6$ and Poisson tail probability bound \eqref{poisson tail probability bound} we have
$$
p_{_{\mathcal{A}_t^c}} \leq e^{-\kappa t} (\frac{e \kappa t}{\rho \kappa t})^{\rho \kappa t}
\leq e^{-\kappa t} e^{- 5 \rho \kappa t},
$$
which also implies $p_{_{\mathcal{A}_t}} \geq 31/32$.\\
Hence by \eqref{Difference of U(t)-Û(t)} we have
\begin{equation}\label{distance of U and U hat for H less than half}
\begin{aligned}
0\leq U(t)-\widehat{U}(t) & \leq  (31/32)^{-1} \exp\{(\rho \kappa)^{1-2H}t- \kappa t/2+e^2 \kappa t/2-5\rho \kappa t/2\} \\
&\sim \mathcal{O}(e^{-t})\,,
\end{aligned}
\end{equation}
where we have used $\rho \geq e^6$ and $\rho \kappa \geq 1$.

So in any case and using $\frac{1}{f(t)}\sim \mathcal{O}(1)$ we have the following inequality
\begin{equation}\label{U is very close to U-hat}
\frac{\widehat{U}(t)}{f(t)}\leq \frac{U(t)}{f(t)} \leq \frac{\widehat{U}(t)}{f(t)}+\mathcal{O}(e^{-t})\,.
\end{equation}

\textbf{Second part: } We would like to show that by constraining ourselves to the integers we do not lose any information on the asymptotic behavior of $\frac{\widehat{U}}{f}$.\\
For any $0<t_1<t_2$ define $\mathcal{C}_{t_1,t_2}$ to be the event that the random walk has no jump on the interval $(t_1,t_2]$. Let $n:=\lfloor t \rfloor$, and for any $x \in \mathds{Z}^d$ denote $\Delta B_{n,t}^x:=B^{x}_t-B^{x}_n$. We have
$$
\begin{aligned}
\widehat{u}(t):= \mathbb{E}^X\bigl[e^{ \int_0^t \mathrm{d}B^{X(s)}_s} \mathbf{1}_{\mathcal{A}_t} \bigr]
&\geq
\mathbb{E}^X\bigl[e^{ \int_0^n \mathrm{d}B^{X(s)}_s} \mathbf{1}_{\mathcal{A}_n} \mathbf{1}_{\mathcal{C}_{n,t}} e^{\int_n^t \mathrm{d}B^{X(s)}_s}\bigr]\\
&=
\mathbb{E}^X\bigl[e^{ \int_0^n \mathrm{d}B^{X(s)}_s} \mathbf{1}_{\mathcal{A}_n} \mathbf{1}_{\mathcal{C}_{n,t}} e^{{\Delta B_{n,t}^{X(n)}}}\bigr]\\
&\geq
\mathbb{E}^X\bigl[e^{ \int_0^n \mathrm{d}B^{X(s)}_s} \mathbf{1}_{\mathcal{A}_n} \mathbf{1}_{\mathcal{C}_{n,t}} \bigr]\,e^{\min_{|x| \leq \mathfrak{N}_n}\Delta B_{n,t}^{x}}\\
&=
\mathbb{E}^X\bigl[e^{ \int_0^n \mathrm{d}B^{X(s)}_s} \mathbf{1}_{\mathcal{A}_n}\bigr] \,  \mathrm{P}^X( \mathcal{C}_{n,t}) \,e^{\min_{|x| \leq \mathfrak{N}_n}\Delta B_{n,t}^{x}}\,.
\end{aligned}
$$
So we have
$$
\widehat{U}(t) = \mathbb{E} \log \widehat{u}(t) \geq \widehat{U}(n)-\kappa(t-n)+ \mathbb{E}\min_{|x| \leq \mathfrak{N}_n}\Delta B_{n,t}^{x}\,.
$$
By elementary probability one can show that expected-value of the maximum of $m$ centered Gaussian random variables is bounded by $\sigma \sqrt{2 \log m}$ where $\sigma^2$ is the maximum of their variances \cite{Lifshits}. As $var(\Delta B_{n,t}^{x})$ is bounded by $1$ for every $x$, we have
$$
\mathbb{E}\max_{|x| \leq \mathfrak{N}_n}\Delta B_{n,t}^{x} \leq \sqrt{2\log \bigl((2\mathfrak{N}_n)^d\bigr)}.
$$
So we have
$$
\widehat{U}(t)\geq \widehat{U}(n)-K_1 \sqrt{\log (n)}\,,
$$
where $K_1$ is a positive constant that only depends on $\kappa$ and $d$ but not on any of the other variables.

We can similarly show that
$$
\widehat{U}(n+1)\geq  \widehat{U}(t)-K_2 \sqrt{\log (t)}\,.
$$
So we have
\begin{equation}\label{quantization sandwich inequality}
\widehat{U}(n)-K \sqrt{\log (n)} \leq  \widehat{U}(t) \leq \widehat{U}(n+1)+K \sqrt{\log (t)}\,.
\end{equation}
Hence the proposition follows from inequalities \eqref{quantization sandwich inequality} and \eqref{U is very close to U-hat}.
\end{proof}
\section{Lipschitz Continuity of Residues of fBm Increments}\label{Lipschitz continuity of residues of fBm increments}

In this section we consider the following stochastic process defined for every $u > n$
$$
Y_n(u):=\int_0^{n} (u-s)^{H-\frac{3}{2}}(\frac{u}{s})^{H-\frac{1}{2}} \, \mathrm{d}W_s\,,
$$
and establish its Lipschitz continuity. This will play a vital role in the succeeding sections. Indeed for $n\in \mathds{N}^{\geq 1}$ and $n+1\leq t_1 < t_2 $ we have
\begin{equation}\label{Fractional Increment Decomposition for Liptschitz chapter}
B_{t_2}-B_{t_1}=\int_0^{n}\Bigl(K_H(t_2,s)-K_H(t_1,s)\Bigr)\mathrm{d}W_s + Z_{n,t_2} \,,
\end{equation}
where $Z_{n,t_2}$ is measurable with respect to the sigma field generated by \\
$\{{W_s-W_n \;;\; s\in [n,t_2]}\}$.\\
Applying the stochastic Fubini's theorem \cite{Protter} to the first right hand side term of \eqref{Fractional Increment Decomposition for Liptschitz chapter} we get
$$
\begin{aligned}
\int_0^{n}\Bigl(K_H(t_2,s)-K_H(t_1,s)\Bigr)\mathrm{d}W_s &=\int_0^{n}\int_{t_1}^{t_2} (u-s)^{H-\frac{3}{2}}(\frac{u}{s})^{H-\frac{1}{2}} \, \mathrm{d}u\; \mathrm{d}W_s\\
&=\int_{t_1}^{t_2}  Y_n(u)\; \mathrm{d}u\,.
\end{aligned}
$$
For $k, n\in \mathds{N}^{\geq 1}$ and $u \in [n+k,n+k+1]$ we define the process $Y_{n,k}$ as $Y_{n,k}(u):=Y_n(u)$.\\
We denote by $\asymp$ and $\curlyeqprec$ respectively, equality and inequality up to a positive constant that only possibly depends on $H$.\\
\begin{prop}\label{Holder continuity of Y for H smaller than half}
Let $k, n\in \mathds{N}^{\geq 1}$ and $u, v \in [n+k,n+k+1]$. Then
\begin{equation}\label{Lipschitz inequality}
\mathbb{E}\bigl[\bigl( Y_{n,k}(u)-Y_{n,k}(v)\bigr)^2\bigr] \curlyeqprec (1+\frac{k}{n})^{2H-1} k^{2H-4}\,(u-v)^2,
\end{equation}
and
\begin{equation}\label{variance bound}
\mathbb{E}\bigl[\bigl(Y_{n,k}(u)\bigr)^2\bigr] \curlyeqprec (1+\frac{k}{n})^{2H-1} k^{2H-2}\,.
\end{equation}
\end{prop}
\begin{proof}
Without loss of generality we may assume that $u\leq v$. Using the It\=o isometry for stochastic integrals \cite{Protter} we have
$$
\begin{aligned}
\mathbb{E}\bigl[\bigl(Y_{n,k}(u)-Y_{n,k}(v)\bigr)^2\bigr] &=\int_0^{n} \Bigl((u-s)^{H-\frac{3}{2}}(\frac{u}{s})^{H-\frac{1}{2}}-
(v-s)^{H-\frac{3}{2}}(\frac{v}{s})^{H-\frac{1}{2}}\Bigr)^2 \, \mathrm{d}s\\
& \leq 2(I_1 +I_2)\,,
\end{aligned}
$$
where
$$
I_1:=\int_0^{n}(\frac{u}{s})^{2H-1}\Bigl((u-s)^{H-\frac{3}{2}}-
(v-s)^{H-\frac{3}{2}}\Bigr)^2 \, \mathrm{d}s\,,
$$
and
$$
I_2:=\int_0^{n} (v-s)^{2H-3}\Bigl((\frac{u}{s})^{H-\frac{1}{2}}-
(\frac{v}{s})^{H-\frac{1}{2}}\Bigr)^2 \, \mathrm{d}s\,.
$$
We furthermore break $I_1$ and $I_2$ into integrals over $[0,\frac{n}{2}]$ and $[\frac{n}{2},n]$ so that $I_1=I_{1a}+I_{1b}$ and $I_2=I_{2a}+I_{2b}$, and we will bound these terms.

Using the following inequality
\begin{equation}\label{Bound on the difference of (u-s)^(3/2)}
|(u-s)^{H-\frac{3}{2}}-
(v-s)^{H-\frac{3}{2}}|\curlyeqprec \frac{v-u}{(u-s)^{\frac{5}{2}-H}}
\end{equation}
which holds for $s<u\leq v$, we have
$$
I_{1b}\curlyeqprec (u-v)^2\,\int_{\frac{n}{2}}^{n}(\frac{u}{s})^{2H-1}(u-s)^{2H-5}\, \mathrm{d}s,
$$
and
$$
I_{1a}\curlyeqprec (u-v)^2 \, u^{2H-1}\,\int_{0}^{\frac{n}{2}}\frac{1}{s^{2H-1}(u-s)^{5-2H}}\, \mathrm{d}s\,.
$$

When $\frac{n}{2}<s$ and $u<n+k+1$, for $H>1/2$ we have
$$
(\frac{u}{s})^{2H-1}\leq (\frac{n+k+1}{n/2})^{2H-1} \curlyeqprec (1+\frac{k}{n})^{2H-1}\,
$$
and for $H\leq 1/2$ we have
$$
(\frac{u}{s})^{2H-1}\leq (\frac{n+k}{n})^{2H-1}\,.
$$
So in any case we have
$$
\begin{aligned}
I_{1b} &\curlyeqprec
(u-v)^2 \,(1+\frac{k}{n})^{2H-1}\int_{\frac{n}{2}}^{n}(u-s)^{2H-5}\, \mathrm{d}s \\
&\curlyeqprec (u-v)^2 (1+\frac{k}{n})^{2H-1} k^{2H-4}.
\end{aligned}
$$

Using $u^{2H-1} \curlyeqprec (n+k)^{2H-1}$ and the inequality $u-s \geq k+n/2 \curlyeqsucc k+n$,  which holds for $s<\frac{n}{2}$, we have
$$
\begin{aligned}
I_{1a}
&\curlyeqprec (u-v)^2 \, (n+k)^{2H-1}(n+k)^{2H-5}\,\int_{0}^{\frac{n}{2}} s^{1-2H}\, \mathrm{d}s\\
&= (u-v)^2 \, (n+k)^{4H-6}\, n^{2-2H} \leq  (u-v)^2 \, (n+k)^{2H-4} \\
&\leq (u-v)^2 \, k^{2H-4}.
\end{aligned}
$$

For $I_{2a}$, we apply $|u^{H-\frac{1}{2}}-
v^{H-\frac{1}{2}}|\curlyeqprec (v-u)u^{H-\frac{3}{2}}$ and notice that for $s\leq n/2$ we have$v-s \geq k+n/2 \curlyeqsucc k+n$. We obtain
$$
\begin{aligned}
I_{2a}
&\curlyeqprec (n+k)^{2H-3}\int_{0}^{\frac{n}{2}}
\Bigl((\frac{u}{s})^{H-\frac{1}{2}}-
(\frac{v}{s})^{H-\frac{1}{2}}\Bigr)^2 \, \mathrm{d}s\\
&\curlyeqprec (n+k)^{2H-3} (u-v)^2 u^{2H-3}\int_{0}^{\frac{n}{2}}
s^{1-2H}\, \mathrm{d}s\\
&\curlyeqprec (u-v)^2 (n+k)^{4H-6} n^{2-2H}\,.
\end{aligned}
$$
So
$$
I_{2a}\curlyeqprec (u-v)^2 (n+k)^{4H-6} (n+k)^{2-2H}\leq (u-v)^2 k^{2H-4}\,.
$$

Similarly for $I_{2b}$ we have
$$
I_{2b}
\curlyeqprec
(u-v)^2 u^{2H-3}\int_{\frac{n}{2}}^{n}s^{1-2H}(v-s)^{2H-3}\,\mathrm{d}s\,.
$$
It can be seen that
$$
\int_{\frac{n}{2}}^{n}s^{1-2H}(v-s)^{2H-3}\,\mathrm{d}s \curlyeqprec  n^{1-2H} \int_{\frac{n}{2}}^{n}(v-s)^{2H-3}\,\mathrm{d}s \curlyeqprec n^{1-2H} k^{2H-2}\,.
$$
So we get
$$
I_{2b} \curlyeqprec (u-v)^2 (1+\frac{k}{n})^{2H-1} k^{2H-4}\,.
$$

Finally for the variance bound \eqref{variance bound}, we similarly have
$$
\begin{aligned}
\mathbb{E}\bigl[\bigl(Y_{n,k}(u)\bigr)^2\bigr]=\int_0^{n} (\frac{u}{s})^{2H-1}(u-s)^{2H-3} \, \mathrm{d}s= \int_0^{\frac{n}{2}} \cdots + \int_{\frac{n}{2}}^{n} \cdots:= J_1+J_2\,,
\end{aligned}
$$
We have
$$
\begin{aligned}
J_1
&\curlyeqprec (n+k)^{2H-3} \int_0^{n/2} (\frac{u}{s})^{2H-1} \, \mathrm{d}s
\curlyeqprec (n+k)^{4H-4} \int_0^{n/2} s^{1-2H} \, \mathrm{d}s\\
& \curlyeqprec (n+k)^{4H-4} n^{2-2H} \curlyeqprec (1+\frac{k}{n})^{2H-2} k^{2H-2}
\end{aligned}
$$
and
$$
J_2
\curlyeqprec (1+\frac{k}{n})^{2H-1} \int_{n/2}^{n} (u-s)^{2H-3} \, \mathrm{d}s\\
\curlyeqprec (1+\frac{k}{n})^{2H-1} k^{2H-2}\,.
$$

\end{proof}
\section{Super-additivity}\label{Super-additivity}
In this section we show that $\{\widehat{U}(n)\}_{n \in \mathds{N}}$ which is not super-additive in the classical sense, still possesses a kind of approximate super-additivity that guarantees the convergence of $\{\frac{\widehat{U}(n)}{n}\}_{n \in \mathds{N}}$.

\begin{thm}\label{limit on integers}
The sequence $\{\frac{\widehat{U}(n)}{n}\}_{n \in \mathds{N}}$ converges to some positive extended real number $\lambda \in [0,+\infty]$.
\end{thm}

While $\{\widehat{U}(n)\}_{n \in \mathds{N}}$ is not super-additive in general as it is in the Brownian motion case, we seek some approximate super-additivity. The almost-super-additivity arguments in \cite{Viens} were the main inspiration for this section.

Let $\{f(n)\}_{n\in \mathds{N}}$ be a sequence of real numbers and $\{\epsilon(n)\}_{n\in \mathds{N}}$ a sequence of non-negative numbers with the property that
$$
\text{(i)} \; \lim_{n\rightarrow \infty} \frac{\epsilon(n)}{n}=0;\qquad
\text{(ii)} \;\sum_{n=1}^{\infty}\frac{\epsilon(2^n)}{2^n}<\infty\,.
$$
Then $\{f(n)\}_{n\in \mathds{N}}$ is called {\it almost super-additive} relative to $\{\epsilon(n)\}_{n\in \mathds{N}}$ if
$$
f(n+m)\geq f(n)+f(m)-\epsilon(n+m)
$$
for any $n, m \in \mathds{N}$. We have the following theorem \cite{Viens, AlmostSuperAdditive}
\begin{thm}\label{almost super additivity limit}
Let $\{f(n)\}_{n\in \mathds{N}}$ be almost super-additive relative to $\{\epsilon(n)\}_{n\in \mathds{N}}$ as defined above.\\
(1) If $\sup_n \frac{f(n)}{n}<+\infty$, then $\lim_{n\rightarrow \infty}\frac{f(n)}{n}$ exists and is finite.\\
(2) If $\sup_n \frac{f(n)}{n}=+\infty$, then $\{\frac{f(n)}{n}\}$ diverges to $+\infty$.
\end{thm}

\begin{lem}
For any $n, m \in \mathds{N}$ we have
$$
\widehat{U}(n+m+1)\geq \widehat{U}(n)+\widehat{U}(m)-c_{\kappa, H} (m+n)^H \sqrt{\log(m+n)}\,.
$$
\end{lem}
\begin{proof}[Proof of Lemma]\hfill \break
\textbf{Step 1: }
Take arbitrary $n, m \in \mathds{N}$ and without loss of generality assume that ${n\geq m}$.\\
Let $\mathcal{A}_n$ be the event that the number of jumps of the random walk in the time interval $[0,n)$ is less than $\mathfrak{N}_n$ defined in \eqref{N_t}, and similarly $\mathcal{B}_m$ be the event that the random walk has less than $\mathfrak{N}_m$ jumps in the interval $[n+1,n+m+1)$. Let also $\mathcal{C}$  be the event that the random walk has no jump in the interval $[n,n+1)$.
We have
\begin{equation}\label{superadditivity equ 1}
\widehat{U}(m+n+1)-\widehat{U}(n) \geq \mathbb{E} \log \mathbb{E}^X \Bigl(\frac{e^{\int_0^n \mathrm{d}B^{X(t)}_t} \mathbf{1}_{\mathcal{A}_n}}{\mathbb{E}^X[e^{\int_0^n \mathrm{d}B^{X(t)}_t} \mathbf{1}_{\mathcal{A}_n}]} e^{\int_n^{n+m+1} \mathrm{d}B^{X(t)}_t}\mathbf{1}_{\mathcal{B}_m\cap\mathcal{C}}\Bigl)\,.
\end{equation}
Let $\mathcal{F}$ be the sigma field generated by the random walk up to time $n$. Then the right-hand-side of the above equation is equal to
\begin{equation}\label{superadditivity equ 2}
\mathbb{E}^X \biggl(\frac{e^{\int_0^n \mathrm{d}B^{X(t)}_t} \mathbf{1}_{\mathcal{A}_n}}{\mathbb{E}^X[e^{\int_0^n \mathrm{d}B^{X(t)}_t} \mathbf{1}_{\mathcal{A}_n}]} \mathbb{E}^X \Bigl(e^{\int_n^{n+m+1} \mathrm{d}B^{X(t)}_t}\mathbf{1}_{\mathcal{B}_m\cap\mathcal{C}}| \mathcal{F}\Bigr)\biggr).
\end{equation}
For any $t\geq n$, let $\widetilde{X}(t):=X(t)-X(n)$. By the Markov property of the random walk, and then the fact that $\{\widetilde{X}(t)\}_{t\geq n}$ is independent of $\mathcal{F}$ we have
$$
\begin{aligned}
\mathbb{E}^X \bigl(e^{\int_n^{n+m+1} \mathrm{d}B^{X(t)}_t}\mathbf{1}_{\mathcal{B}_m\cap\mathcal{C}}|\mathcal{F}\bigr)&=
\mathbb{E}^X \bigl(e^{\int_n^{n+m+1} \mathrm{d}B^{X(t)}_t}\mathbf{1}_{\mathcal{B}_m\cap\mathcal{C}}|X(n)\bigr)\\
&=\mathbb{E}^X \bigl(e^{\int_n^{n+m+1} \mathrm{d}B^{\widetilde{X}(t)+X(n)}_t}\mathbf{1}_{\mathcal{B}_m\cap\mathcal{C}}|X(n)\bigr)\\
&=\mathbb{E}^{\widetilde{X}} \bigl(e^{\int_n^{n+m+1} \mathrm{d}B^{\widetilde{X}(t)+X(n)}_t}\mathbf{1}_{\mathcal{B}_m\cap\mathcal{C}}\bigr)\\
&=\mathbb{E}^{\widetilde{X}} \bigl(e^{\int_n^{n+m+1} \mathrm{d}B^{\widetilde{X}(t)+Y}_t}\mathbf{1}_{\mathcal{B}_m\cap\mathcal{C}}\bigr),\\
\end{aligned}
$$
where $Y:=X(n)$.\\
Now denote by $\overline{\mathbb{E}}^Y$ the expectation with respect to the random variable Y with the following distribution
$$
\overline{\mathrm{P}}(Y=y)=\mathbb{E}^X \Bigl( \frac{e^{\int_0^n \mathrm{d}B^{X(t)}_t} \mathbf{1}_{\mathcal{A}_n}}{\mathbb{E}^X[e^{\int_0^n \mathrm{d}B^{X(t)}_t} \mathbf{1}_{\mathcal{A}_n}]}\mathbf{1}_{X(n)=y}\Bigr) \qquad : \qquad y\in\mathds{Z}^d.
$$
So equations \eqref{superadditivity equ 1} and \eqref{superadditivity equ 2} imply
\begin{equation}\label{superaditivity equ 3}
\begin{aligned}
\widehat{U}(m+n+1)-\widehat{U}(n) &\geq \mathbb{E} \,\log \, \overline{\mathbb{E}}^Y \Bigl( \mathbb{E}^{\widetilde{X}} \bigl(e^{\int_n^{n+m+1} \mathrm{d}B^{\widetilde{X}(t)+Y}_t}\mathbf{1}_{\mathcal{B}_m\cap\mathcal{C}}\bigr)\Bigr)\\
&\geq \mathbb{E} \,\overline{\mathbb{E}}^Y \log \, \mathbb{E}^{\widetilde{X}} \bigl(e^{\int_n^{n+m+1} \mathrm{d}B^{\widetilde{X}(t)+Y}_t}\mathbf{1}_{\mathcal{B}_m\cap\mathcal{C}}\bigr).
\end{aligned}
\end{equation}
\textbf{Step 2: }
Let $\{\widehat{W}^{x}\}_{x \in \mathds{Z}^d}$ be a family of independent standard Brownian motions, which is independent of any random variable introduced so far, in particular independent of the random walks $X(\cdot)$ and $\widetilde{X}(\cdot)$, the fractional Brownian motions $\{B^x\}_{x \in \mathds{Z}^d}$ and hence their corresponding Brownian motions $\{W^x\}_{x \in \mathds{Z}^d}$ appearing in their integral representation. For any $x \in \mathds{Z}^d$ define $\widetilde{W}_s^{x}$ as
$$
\widetilde{W}_t^{x}:=
\begin{cases}
\widehat{W}^{x}_t & \text{for } 0 \leq t \leq n \\
W^x_t- W^x_{n} + \widehat{W}^{x}_{n}  & \text{for } t > n \,.
\end{cases}
$$
It is easily verified that $\widetilde{W}^{x}$ is itself a standard Brownian motion.\\
We define the following family of fractional Brownian motions indexed by $\mathds{Z}^d$
\begin{equation}\label{Definition of artificial B for super additivity}
\widetilde{B}^{x}_t:= \int_0^t K_H(t,s) \mathrm{d}\widetilde{W}^{x}_s\,.
\end{equation}
It is clear that for $t \geq n$
\begin{equation}\label{Decomposition of B for super-additivity}
\widetilde{B}^{x}_t = \int_0^{n} K_H(t,s)\mathrm{d}\widehat{W}^{x}_s
+ \int_{n}^t K_H(t,s)\mathrm{d}W^x_s\,.
\end{equation}
Now let $\{t_i\}_i$, $t_i \geq n+1$, be the jump times of the random walk after time $t=n+1$, and for every $i$ let $x_i$ be the position of the random walk in the time interval $[t_i,t_{i+1})$. Then by \eqref{Decomposition of B for super-additivity} and noting $B^{x}_t = \int_{0}^t K_H(t,s)\mathrm{d}W^x_s$ we have
$$
\int_{n+1}^{n+m+1} \mathrm{d}B^{\widetilde{X}(t)+Y}_t=\int_{n+1}^{n+m+1} \mathrm{d}\widetilde{B}^{\widetilde{X}(t)+Y}_t+\Delta^X\,,
$$
where
$$
\begin{aligned}
\Delta^X:=&\sum_i \int_0^n \bigl(K_H(s,t_{i+1})-K_H(s,t_{i})\bigr)\mathrm{d}W^{x_i}_s\\
& \qquad \qquad -\sum_i \int_0^n \bigl(K_H(s,t_{i+1})-K_H(s,t_{i})\bigr)\mathrm{d}\widetilde{W}^{x_i}_s\,.
\end{aligned}
$$
By the definition of $K_H$ and using the stochastic Fubini we have
$$
\begin{aligned}
\int_0^n \bigl(K_H(s,t_{i+1})-K_H(s,t_{i})\bigr)\mathrm{d}W^{x_i}_s &=
c_H \int_0^n \int_{t_i}^{t_{i+1}} (u-s)^{H-\frac{3}{2}}(\frac{u}{s})^{H-\frac{1}{2}}\,\mathrm{d}u\,\mathrm{d}W^{x_i}_s\\
&=c_H \int_{t_i}^{t_{i+1}} \int_0^n (u-s)^{H-\frac{3}{2}}(\frac{u}{s})^{H-\frac{1}{2}}\,\mathrm{d}W^{x_i}_s \, \mathrm{d}u\\
&=c_H \int_{t_i}^{t_{i+1}} Y_n^{x_i}(u) \, \mathrm{d}u\,,
\end{aligned}
$$
and similarly
$$
\int_0^n \bigl(K_H(s,t_{i+1})-K_H(s,t_{i})\bigr)\mathrm{d}\widetilde{W}^{x_i}_s =
c_H \int_{t_i}^{t_{i+1}} \widetilde{Y}_n^{x_i}(u) \, \mathrm{d}u\,,
$$
where
$$
\widetilde{Y}_n^{x_i}(u)=\int_0^n (u-s)^{H-\frac{3}{2}}(\frac{u}{s})^{H-\frac{1}{2}}\,\mathrm{d}\widetilde{W}^{x_i}_s\,.
$$
Hence we have
$$
\Delta^X=c_H \int_{n+1}^{n+m+1} Y_n^{X(u)}(u) \, \mathrm{d}u-c_H \int_{n+1}^{n+m+1} \widetilde{Y}_n^{X(u)}(u) \, \mathrm{d}u\,.
$$
So under the event $\mathcal{A}_n\cap\mathcal{B}_m$, $\Delta^X$ is bounded from below as follows
$$
\Delta^X \geq c_H \sum_{k=1}^{m} \inf_{\substack{|x|\leq \mathfrak{N}_n+\mathfrak{N}_m\\
u\in [n+k,n+k+1]}}Y_n^{x}(u)-c_H \sum_{k=1}^{m} \sup_{\substack{|x|\leq \mathfrak{N}_n+\mathfrak{N}_m\\
u\in [n+k,n+k+1]}}\widetilde{Y}_n^{x}(u)\,.
$$
Also, under the event $\mathcal{A}_n\cap\mathcal{C}$ we have
$$
\begin{aligned}
\int_{n}^{n+1}\mathrm{d}B^{\widetilde{X}(t)+Y}_t &= B^{Y}_{n+1}-B^{Y}_{n}\\
& \geq \inf_{|y|\leq \mathfrak{N}_n} (B^{y}_{n+1}-B^{y}_{n})\,.
\end{aligned}
$$
So under the event $\mathcal{A}_n\cap\mathcal{C}\cap\mathcal{B}_m$ we have
$$
\begin{aligned}
&\int_n^{n+m+1} \mathrm{d}B^{\widetilde{X}(t)+Y}_t =
\int_n^{n+1} \mathrm{d}B^{\widetilde{X}(t)+Y}_t+\int_{n+1}^{n+m+1} \mathrm{d}\widetilde{B}^{\widetilde{X}(t)+Y}_t+\Delta^X\\
&\qquad \qquad \geq \int_{n+1}^{n+m+1} \mathrm{d}\widetilde{B}^{\widetilde{X}(t)+Y}_t +\inf_{|y|\leq \mathfrak{N}_n} (B^{y}_{n+1}-B^{y}_{n}) \\
&\qquad \qquad \qquad +c_H \sum_{k=1}^{m} \inf_{\substack{|x|\leq \mathfrak{N}_n+\mathfrak{N}_m\\
u\in [n+k,n+k+1]}}Y_n^{x}(u)-c_H \sum_{k=1}^{m} \sup_{\substack{|x|\leq \mathfrak{N}_n+\mathfrak{N}_m\\
u\in [n+k,n+k+1]}}\widetilde{Y}_n^{x}(u)\,.
\end{aligned}
$$
\textbf{Step 3: }
Plugging this inequality into Equation \eqref{superaditivity equ 3} we get
$$
\begin{aligned}
\widehat{U}(m+n+1)-\widehat{U}(n) &\geq
\mathbb{E} \,\overline{\mathbb{E}}^Y \log \, \mathbb{E}^{\widetilde{X}} \bigl(e^{\int_{n+1}^{n+m+1} \mathrm{d}\widetilde{B}^{\widetilde{X}(t)+Y}_t}\mathbf{1}_{\mathcal{B}_m\cap\mathcal{C}}\bigr)\\
&\quad +\mathbb{E}\inf_{|y|\leq \mathfrak{N}_n} (B^{y}_{n+1}-B^{y}_{n})
+c_H \sum_{k=1}^{m} \mathbb{E}\inf_{\substack{|x|\leq \mathfrak{N}_n+\mathfrak{N}_m\\
u\in [n+k,n+k+1]}}Y_n^{x}(u)\\
&\qquad -c_H \mathbb{E}\sum_{k=1}^{m} \sup_{\substack{|x|\leq \mathfrak{N}_n+\mathfrak{N}_m\\
u\in [n+k,n+k+1]}}\widetilde{Y}_n^{x}(u)
\end{aligned}
$$
For $t\geq n+1$, let $X'(t):=X(t)-X(n+1)$. Then we have
$$
\begin{aligned}
&\mathbb{E} \,\overline{\mathbb{E}}^Y \log \, \mathbb{E}^{\widetilde{X}} \bigl(e^{\int_{n+1}^{n+m+1} \mathrm{d}\widetilde{B}^{\widetilde{X}(t)+Y}_t}\mathbf{1}_{\mathcal{B}_m\cap\mathcal{C}}\bigr)\\
&\qquad =\mathbb{E} \,\overline{\mathbb{E}}^Y \log \, \mathbb{E}^{\widetilde{X}} \bigl(e^{\int_{n+1}^{n+m+1} \mathrm{d}\widetilde{B}^{X'(t)+Y}_t}\mathbf{1}_{\mathcal{B}_m\cap\mathcal{C}}\bigr)\\
&\qquad =\mathbb{E} \,\overline{\mathbb{E}}^Y \log \, \mathbb{E}^{X'} \bigl(e^{\int_{n+1}^{n+m+1} \mathrm{d}\widetilde{B}^{X'(t)+Y}_t}\mathbf{1}_{\mathcal{B}_m}\bigr)+\log \mathrm{P}(\mathcal{C})\,.
\end{aligned}
$$
Let $\widehat{\mathcal{G}}_{[0,n]}$ be the sigma field generated
by $\{\widehat{W}^{x}_s \;;\; s \in [0,n] \,,\, x \in \mathds{Z}^d\}$ and $\mathcal{G}_{[n,\infty)}$ the sigma field generated by $\{{W^{x}_s- W^{x}_{n}}\;;\; s \in [n,\infty)\,,\, x \in \mathds{Z}^d\}$. Also denote by $\mathcal{G}_o$ the sigma field generated by $\{W^{x}_s\;;\; s \in [0,n]\,,\, x \in \mathds{Z}^d\}$. Clearly $\mathcal{G}_1$ is independent of $\mathcal{G}_o$.
It is evident that for any $t \geq n$ the process $\widetilde{B}^{x}_t$ is measurable with respect to $\mathcal{G}_1:=\widehat{\mathcal{G}}_{[0,n]} \vee \mathcal{G}_{[n,\infty)}$ where $\vee$ denotes the smallest sigma field containing the both. So $\int_{n+1}^{n+m+1} \mathrm{d}\widetilde{B}^{ \widetilde{X}(t)+y}_t$ is also measurable with respect to $\mathcal{G}_1$.\\
Using the notation $f(Y):=\log \, \mathbb{E}^{X'} \bigl(e^{\int_{n+1}^{n+m+1} \mathrm{d}\widetilde{B}^{X'(t)+Y}_t}\mathbf{1}_{\mathcal{B}_m}\bigr)$, we have
$$
\begin{aligned}
\mathbb{E} \,\overline{\mathbb{E}}^Y [f(Y)]&=\mathbb{E} \mathbb{E}^X \bigl[\frac{e^{\int_0^n \mathrm{d}B^{X(t)}_t} \mathbf{1}_{\mathcal{A}_n}}{\mathbb{E}^X[e^{\int_0^n \mathrm{d}B^{X(t)}_t} \mathbf{1}_{\mathcal{A}_n}]}f(Y)\bigr]\\
&=\mathbb{E}^X \bigl[\mathbb{E} \bigl(\frac{e^{\int_0^n \mathrm{d}B^{X(t)}_t} \mathbf{1}_{\mathcal{A}_n}}{\mathbb{E}^X[e^{\int_0^n \mathrm{d}B^{X(t)}_t} \mathbf{1}_{\mathcal{A}_n}]}\bigl) \mathbb{E} \bigl(f(Y)\bigr)\bigr]\,.
\end{aligned}
$$
where we used the fact that $\frac{e^{\int_0^n \mathrm{d}B^{X(t)}_t} \mathbf{1}_{\mathcal{A}_n}}{\mathbb{E}^X[e^{\int_0^n \mathrm{d}B^{X(t)}_t} \mathbf{1}_{\mathcal{A}_n}]}$ is measurable with respect to $\mathcal{G}_o$ and hence independent of $f(Y)$.
But for every $y\in\mathds{Z}^d$, the random variable $f(y)=\mathbb{E}^{X'} \bigl(e^{\int_{n+1}^{n+m+1} \mathrm{d}\widetilde{B}^{\widetilde{X}(t)+y}_t}\mathbf{1}_{\mathcal{B}_m}\bigr)$ has the same distribution as $\mathbb{E}^{X} \bigl(e^{\int_{0}^{m} \mathrm{d}B^{X(t)}_t}\mathbf{1}_{\mathcal{A}_m}\bigr)$. So we have
$$
\mathbb{E} [f(Y)]=\mathbb{E}\log \, \mathbb{E}^{X} \bigl(e^{\int_{0}^{m} \mathrm{d}B^{X(t)}_t}\mathbf{1}_{\mathcal{A}_m}\bigr)=\widehat{U}(m)\,.
$$
Hence we get the following conclusion
\begin{equation}\label{superadditivity inequality, general form}
\widehat{U}(m+n+1)-\widehat{U}(n) \geq \widehat{U}(m) - \widehat{\epsilon}(n,m)\,,
\end{equation}
where
\begin{equation}\label{expression for epsilon-hat}
\begin{aligned}
\widehat{\epsilon}(n,m)&:=-\mathbb{E}\inf_{|y|\leq \mathfrak{N}_n} (B^{y}_{n+1}-B^{y}_{n})
-c_H \sum_{k=1}^{m} \mathbb{E}\inf_{\substack{|x|\leq \mathfrak{N}_n+\mathfrak{N}_m\\
u\in [n+k,n+k+1]}}Y_n^{x}(u)\\
&\qquad -\log \mathrm{P}(\mathcal{C})+c_H \mathbb{E}\sum_{k=1}^{m} \sup_{\substack{|x|\leq \mathfrak{N}_n+\mathfrak{N}_m\\
u\in [n+k,n+k+1]}}\widetilde{Y}_n^{x}(u)\\
&=\mathbb{E}\sup_{|y|\leq \mathfrak{N}_n} (B^{y}_{n+1}-B^{y}_{n})
+c_H \sum_{i=1}^{m} \mathbb{E}\sup_{\substack{|x|\leq \mathfrak{N}_n+\mathfrak{N}_m\\
u\in [n+k,n+k+1]}}Y_n^{x}(u)\\
&\qquad -\log \mathrm{P}(\mathcal{C})+c_H \mathbb{E}\sum_{k=1}^{m} \sup_{\substack{|x|\leq \mathfrak{N}_n+\mathfrak{N}_m\\
u\in [n+k,n+k+1]}}\widetilde{Y}_n^{x}(u)\,.
\end{aligned}
\end{equation}
\textbf{Step 4: }
We are going to bound the terms in \eqref{expression for epsilon-hat} applying Dudley's theorem \ref{Dudley theorem}.\\
For the first term we use the fact from elementary probability \cite{Lifshits} that expected-value of the maximum of $n$ centered Gaussian random variables is bounded by $\sigma \sqrt{2 \log n}$ where $\sigma^2$ is the maximum of their variances . As $var(B^{y}_{n+1}-B^{y}_{n})=1$ for any $y$ and $n$, we have
\begin{equation}\label{the first term for epsilon-hat}
\mathbb{E}\sup_{|y|\leq \mathfrak{N}_n} (B^{y}_{n+1}-B^{y}_{n})\leq \sqrt{2\log \mathfrak{N}_n}\leq c'_{\kappa, H}\sqrt{\log n}\,,
\end{equation}
where $c'_{\kappa, H}$ is a positive constant that only depends on $\kappa$ and $H$.

For $l\in \mathds{N}^{\geq 1}$, let $\{u_i\}_{i=1}^l$ be the $l$ equally-spaced points on the interval $(n+k,n+k+1)$. Then for any $u \in [n+k,n+k+1]$ there exists a $u_i$ with $|u-u_i|\leq \frac{1}{2l}$. Using proposition \ref{Holder continuity of Y for H smaller than half} on the Lipschitz continuity of $Y_n$ and noting that $k\leq m\leq n$, for every $x \in \mathds{Z}^d$ we have
$$
\mathbb{E}\bigl[Y_{n}^x(u)-Y_{n}^x(u_i)\bigr]^2 \leq c_H k^{2H-4}\,(u-u_i)^2 \leq c_H k^{2H-4} \frac{1}{(2l)^2}
$$
and
$$
\mathbb{E}\bigl(Y_{n}^x(u)\bigr)^2 \leq C_H k^{2H-2}\,,
$$
where $c_H$ and $C_H$ are some positive constants that depend only on $H$. This means that for $0<\varepsilon<c'_H k^{H-2}$, where $c'_H:=\sqrt{c_H}/2$, we can cover
$$
{\{Y_n^x(u)\;;\;u\in [n+k,n+k+1], x\in \mathds{Z}^d, |x|\leq \mathfrak{N}_{n}+\mathfrak{N}_{m}\}}
$$
by $(\mathfrak{N}_{n}+\mathfrak{N}_{m})\frac{c'_H k^{H-2}}{\varepsilon}$ balls of radius $\varepsilon$.\\
For ${c'_H k^{H-2} \leq \varepsilon < C'_H k^{H-1}}$, where $C'_H:=\sqrt{2C_H}$, this set can be covered by $\mathfrak{N}_{n}+\mathfrak{N}_{m}$ $\varepsilon$-balls. And finally for ${\varepsilon \geq C'_H k^{H-1}}$, the whole set can be cover with one single ball. So by Dudley's theorem \ref{Dudley theorem} we have
$$
\begin{aligned}
\mathbb{E}\sup_{\substack{|x|\leq \mathfrak{N}_n+\mathfrak{N}_m\\
u\in [n+i,n+i+1]}}Y_n^{x}(u)
& \leq K \int_{0}^{c'_H k^{H-2}} \sqrt{\log \bigl((\mathfrak{N}_{n}+\mathfrak{N}_{m})\frac{c'_H k^{H-2}}{\varepsilon}\bigr)} \,\mathrm{d}\varepsilon\\
& \qquad \qquad +K \int_{c'_H k^{H-2}}^{C'_H k^{H-1}} \sqrt{\log (\mathfrak{N}_{n}+\mathfrak{N}_{m})} \,\mathrm{d}\varepsilon\\
& \leq k^{H-1}c_{\kappa, H}'' \sqrt{\log (n+m)}\,.
\end{aligned}
$$
So
\begin{equation}\label{the second term for epsilon-hat}
\begin{aligned}
\sum_{k=1}^{m}\mathbb{E}\sup_{\substack{|x|\leq \mathfrak{N}_n+\mathfrak{N}_m\\
u\in [n+i,n+i+1]}}Y_n^{x}(u)
&\leq c_{\kappa, H}'' \sqrt{\log (n+m)} \sum_{k=1}^{m} k^{H-1}\\
& \leq c_{\kappa, H}'' m^H \sqrt{\log (n+m)}\,.
\end{aligned}
\end{equation}
In the same way we have
\begin{equation}\label{the third term for epsilon-hat}
\sum_{k=1}^{m}\mathbb{E}\sup_{\substack{|x|\leq \mathfrak{N}_n+\mathfrak{N}_m\\
u\in [n+i,n+i+1]}}\widetilde{Y}_n^{x}(u)
\leq c_{\kappa, H}'' m^H \sqrt{\log (n+m)}\,.
\end{equation}
As we additionally have $\mathrm{P}(\mathcal{C})=e^{-\kappa}$, by Equations \eqref{expression for epsilon-hat}, \eqref{the first term for epsilon-hat}, \eqref{the second term for epsilon-hat}, and \eqref{the third term for epsilon-hat} we obtain the following inequality
$$
\begin{aligned}
\widehat{\epsilon}(n,m)
\leq c_{\kappa, H} m^H \sqrt{\log (n+m)}\,.
\end{aligned}
$$
This inequality along with Equation \eqref{superadditivity inequality, general form} completes the proof.
\end{proof}

\begin{proof}[Proof of Theorem \ref{limit on integers}]
Applying the above lemma we can easily see that $\{\widehat{U}(n-1)\}_{n \in \mathds{N}}$ is almost-super-additive with respect to $\epsilon(n):=c_{\kappa, H} n^H \sqrt{\log(n)}$.
Then theorem \ref{almost super additivity limit} implies that $\{\frac{\widehat{U}(n-1)}{n}\}_{n \in \mathds{N}}$ converges to some positive extended real number and hence so does $\{\frac{\widehat{U}(n)}{n}\}_{n \in \mathds{N}}$.
\end{proof}
\section{Quenched Limits}\label{Quenched limits}
In this section we consider the quenched limits. \\
We introduce the following notation:
$$
\widehat{u}(t):= \mathbb{E}^X\bigl[e^{ \int_0^t \mathrm{d}B^{X(s)}_s} \mathbf{1}_{\mathcal{A}_t} \bigr] \quad \text{and} \quad \widehat{U}(t):=\mathbb{E} \log \widehat{u}(t)
$$
where $\mathcal{A}_t$ is the same event defined after Eq. \eqref{N_t}. Recall also the definition of $u(t)$ from Eq. \eqref{u}.

In the first proposition we show that the convergence of $\{\frac{\widehat{U}(n)}{n}\}_{n \in \mathds{N}}$ to a strictly positive number $\lambda$ implies the convergence of $\{\frac{\log \widehat{u}(n)}{n}\}_{n \in \mathds{N}}$ to $\lambda$. Then in the second proposition, we show that this in its turn implies the convergence of $\{\frac{\log u(t)}{t}\}_{t\in \mathds{R}^{>0}}$ to $\lambda$ as $t$ goes off to $+\infty$. In the following proof we use arguments from the Malliavin calculus. The use of Malliavin calculus to obtain concentration in the polymer literature has appeared in earlier publications; see for examples \cite{Rovira-Tindel, CranstonComets}.
\begin{prop}
For any function $f:\mathds{R}^{>0}\rightarrow \mathds{R}^{>0}$ that grows at least as fast as a linear function, we have
$$
\lim_{\substack{n\rightarrow \infty \\ n \in \mathds{N}}}\Bigl(\frac{\widehat{U}(n)}{f(n)}-\frac{\log \widehat{u}(n)}{f(n)}\Bigr)=0\qquad \text{almost surely}.
$$
\end{prop}
\begin{proof}
We will apply theorem \ref{Ustunel's thm} which provides concentration bounds on Malliavin derivable random variables.\\
For $X(\cdot)$, an arbitrary but fixed sample path of the random walk and $t \in \mathds{R}$, let $g_{t}^X: \mathds{R}\times \mathds{Z}^d \longrightarrow \mathds{R}$ be the function defined as
$$
g_{t}^X (s,x):= \mathbf{1}_{[0,t]}(s)\, \mathbf{1}_{X(s)}(x).
$$
With the notions introduced in Section \ref{preliminaries} it can be easily seen that $g_{t}^X$ is in $\mathcal{H}$ and moreover
$$
\textbf{B}(g_{t}^X)= \int_0^t \mathrm{d}B^{X(s)}_s\,,
$$
which shows that
$$
\nabla \int_0^t \mathrm{d}B^{X(s)}_s=g_{t}^X.
$$
Hence we have
$$
\nabla \widehat{u}(n)= \mathbb{E}^X\bigl[e^{ \int_0^n \mathrm{d}B^{X(s)}_s} \mathbf{1}_{\mathcal{A}_n} \, g_{n}^X\bigr]
$$
and
$$
\nabla \Bigl(\log \widehat{u}(n)\Bigr)= \frac{1}{\widehat{u}(n)} \nabla \widehat{u}(n)=\frac{1}{\widehat{u}(n)}\mathbb{E}^X\bigl[e^{ \int_0^n \mathrm{d}B^{X(s)}_s} \mathbf{1}_{\mathcal{A}_n} \, g_{n}^X\bigr]\,.
$$
For $X_1(\cdot)$ and $X_2(\cdot)$, independent random walks having the same law as $X(\cdot)$, we have
$$
\begin{aligned}
||\nabla \widehat{u}(n)||_{\mathcal{H}}^2
&=\Bigl\langle\mathbb{E}^X\bigl[e^{ \int_0^n \mathrm{d}B^{X(s)}_s} \mathbf{1}_{\mathcal{A}_n} \, g_{n}^X\bigr]\,,\,\mathbb{E}^X\bigl[e^{ \int_0^n \mathrm{d}B^{X(s)}_s} \mathbf{1}_{\mathcal{A}_n} \, g_{n}^X\bigr] \Bigr\rangle_\mathcal{H}\\
&=\Bigl\langle\mathbb{E}^{X_1}\bigl[e^{ \int_0^n \mathrm{d}B^{{X_1}(s)}_s} \mathbf{1}_{\mathcal{A}^1_n} \, g_{n}^{X_1}\bigr]\,,\,\mathbb{E}^{X_2}\bigl[e^{ \int_0^n \mathrm{d}B^{{X_2}(s)}_s} \mathbf{1}_{\mathcal{A}^2_n} \, g_{n}^{X_2}\bigr]
\Bigr\rangle_\mathcal{H}\\
&=\mathbb{E}^{X_1}\mathbb{E}^{X_2}\Bigl[e^{ \int_0^n \mathrm{d}B^{{X_1}(s)}_s} \mathbf{1}_{\mathcal{A}^1_n}e^{ \int_0^n \mathrm{d}B^{{X_2}(s)}_s} \mathbf{1}_{\mathcal{A}^2_n}\,\langle  g_{n}^{X_1}, \, g_{n}^{X_2}
\rangle_\mathcal{H}\Bigr]\\
&\leq \mathbb{E}^{X_1}\mathbb{E}^{X_2}\Bigl[e^{ \int_0^n \mathrm{d}B^{{X_1}(s)}_s} \mathbf{1}_{\mathcal{A}^1_n}e^{ \int_0^n \mathrm{d}B^{{X_2}(s)}_s} \mathbf{1}_{\mathcal{A}^2_n}\, || g_{n}^{X_1}||_\mathcal{H}\,  ||g_{n}^{X_2}||
_\mathcal{H}\Bigr]\\
&\leq \biggl(\mathbb{E}^{X}\Bigl(e^{ \int_0^n \mathrm{d}B^{{X}(s)}_s} \mathbf{1}_{\mathcal{A}_n} || g_{n}^{X}||_\mathcal{H}\Bigr)\biggr)^2.
\end{aligned}
$$
But we have
$$
|| g_{n}^{X}||_\mathcal{H}^2=\mathbb{E}\bigl(\int_0^n \mathrm{d}B^{{X}(s)}_s\bigr)^2.
$$
So for $H>1/2$ we have
$$
|| g_{n}^{X}||_\mathcal{H}^2\leq n^{2H}\,,
$$
and for $H\leq 1/2$ and under $\mathcal{A}_n$
$$
|| g_{n}^{X}||_\mathcal{H}^2\leq \mathfrak{N}_n (\frac{n}{\mathfrak{N}_n})^{2H}\leq n\,(\rho \kappa)^{1-2H}\,.
$$
The fact that $|| g_{n}^{X}||_\mathcal{H}$ has an upper bound that doesn't depend on the random walk leads to the following bound
$$
||\nabla \Bigl(\log \widehat{u}(n)\Bigr)||^2\leq || g_{n}^{X}||_\mathcal{H}^2.
$$
So by theorem \ref{Ustunel's thm} we have
$$
\mathrm{P}\Bigl(\bigl| \log \widehat{u}(n)-\widehat{U}(n)\bigr|> 2n^H \sqrt{\log n}\Bigr)\leq 2e^{-2\log n}=2n^{-2}.
$$
As the right-hand-side of this inequality is summable we can apply Borel-Cantelli lemma to conclude that almost surely there exists $N$ such that for any $n\in \mathds{N}$ with $n\geq N$ we have
$$
\bigl|\log \widehat{u}(n)-\widehat{U}(n)\bigr|\leq 2n^H \sqrt{\log n}\,,
$$
which along with the assumption on the growth rate of $f(\cdot)$ implies the almost sure limit
$$
\lim_{n \rightarrow \infty} \frac{\log \widehat{u}(n)}{f(n)}-\frac{\widehat{U}(n)}{f(n)}=0\qquad \text{almost surely}.
$$
\end{proof}

\begin{prop}
For any real positive function $f:\mathds{R}^{>0}\rightarrow \mathds{R}^{>0}$ which satisfies $\alpha|s-t|\leq|f(s)-f(t)|\leq \beta|s-t|^p$ for some fixed positive numbers $\alpha$, $\beta$ and $p$, we have
$$
\limsup_{t\rightarrow \infty}\frac{\log u(t)}{f(t)}=\limsup_{\substack{n\rightarrow \infty \\ n \in \mathds{N}}}\frac{\log \widehat{u}(n)}{f(n)}\qquad \text{almost surely},
$$
and
$$
\liminf_{t\rightarrow \infty}\frac{\log u(t)}{f(t)}=\liminf_{\substack{n\rightarrow \infty \\ n \in \mathds{N}}}\frac{\log \widehat{u}(n)}{f(n)}\qquad \text{almost surely}.
$$
\end{prop}
\begin{proof}\hfill \break
\textbf{Step 1: }
For $l,n\in \mathds{N}^{\geq 1}$, let $\{t_i\}_{i=1}^{l}$ be the $l$ uniformly spaced points on the interval $(n-1,n)$. It is evident that for any $x\in \mathds{Z}^d$ and for any $t \in [n-1,n]$, there exists a $t_i$ with $|t-t_i|\leq \frac{1}{2l}$. Then we have
$$
\mathbb{E}\Bigl((B_{t}^x-B_{n}^x)-(B_{t_i}^x-B_{n}^x) \Bigr)^2=
\mathbb{E}\Bigl(B_{t}^x-B_{t_i}^x \Bigr)^2=\frac{1}{(2l)^{2H}}\,.
$$
So for $0<\varepsilon<2^{-H}$ we can cover the set $\{B_{t}^x-B_{n}^x \; ;\;\ t\in [n-1,n]\}$ by $l=\frac{1}{2\varepsilon^{1/H}}$ $\varepsilon$-balls and for $2^{-H} \leq \varepsilon$ the whole set can be covered by a single element. So by Dudley's theorem we have
$$
\mathbb{E} \Bigl(\sup_{n-1\leq t \leq n}(B_{t}^x-B_{n}^x) \Bigr) \leq K \int_{0}^{2^{-H} }\sqrt{\log \frac{1}{2\varepsilon^{1/H}}}=K_1\,,
$$
where $K$ and $K_1$ are some universal constants.\\
We also have $\mathbb{E}(B_{t}^x-B_{n}^x )^2 \leq 1$ for every $t\in[n-1,n]$. So by Borell's inequality \ref{Borell}, for any $k\in \mathds{N}$ and any $n$ large enough we have
$$
\begin{aligned}
&\mathrm{P}\Bigl(\sup_{n-1\leq t \leq n}(B_{t}^x-B_{n}^x)\geq (k+2)(d+1)\log n \,\Bigr)\\
&\qquad \qquad \qquad \leq e^{-2(k+2)(d+1)\log n}=n^{-2(k+2)(d+1)}\,.
\end{aligned}
$$
So
$$
\begin{aligned}
&\mathrm{P}\Bigl(\bigcup_{|x|\leq \mathfrak{N}_n n^k}\{\sup_{n-1\leq t \leq n}(B_{t}^x-B_{n}^x)\geq (k+2)(d+1)\log n\} \,\Bigr)\\
&\qquad \qquad \qquad \qquad \leq (2\mathfrak{N}_n n^k+1)^{d}n^{-(k+2)(d+1)}\leq n^{-(k+2)}\,,
\end{aligned}
$$
and hence
$$
\begin{aligned}
&\mathrm{P}\Bigl(\bigcup_{k\in \mathds{N}}\bigcup_{|x|\leq \mathfrak{N}_n n^k}\{\sup_{n-1\leq t \leq n}(B_{t}^x-B_{n}^x)\geq (k+2)(d+1)\log n\} \,\Bigr)\\
&\qquad \qquad \qquad \qquad \leq \sum_{k} n^{-(k+2)} \leq 2n^{-2}\,.
\end{aligned}
$$
By Borel-Cantelli lemma, almost surely there exists $N_1$ such that for any $n\geq N_1$ and for every $k\in \mathds{N}$ we have
$$
\sup_{|x|\leq \mathfrak{N}_n n^k}\sup_{n-1\leq t \leq n}(B_{t}^x-B_{n}^x)\leq (k+2)(d+1)\log n\,
$$
which is equivalent to
\begin{equation}\label{lower bound for Bn-Bt for x in Nn*n^k}
\inf_{|x|\leq \mathfrak{N}_n n^k}\inf_{n-1\leq t \leq n}(B_{n}^x-B_{t}^x)\geq -(k+2)(d+1)\log n\,.
\end{equation}
Using the same procedure we can easily show that almost surely there exists $N_2$ such that for any $n\geq N_2$ we have
\begin{equation}\label{lower bound for Bt-B(n-1) for x in Nn}
\inf_{|x|\leq \mathfrak{N}_n}\inf_{n-1\leq t \leq n}(B_{t}^x-B_{n-1}^x)\geq -\log n\,.
\end{equation}
\textbf{Step 2: }
For any given $t \in \mathds{R}^{>0}$ and $k \in \mathds{N}$, let $n:=\lceil t\rceil$ (the ceiling of $t$), i.e. the smallest integer not larger than or equal to $t$, and define $\mathcal{A}_{t,k}$ as the event that the number of jumps of the random walk on $[0,t]$ is larger than or equal to $\mathfrak{N}_{n} n^k$ but strictly less than $\mathfrak{N}_{n} n^{k+1}$. We use the following notations
\begin{equation}\label{definition of u-hat-k-t}
\widehat{u}_k(t):=\mathbb{E}^X\Bigl[e^{ \int_0^t \mathrm{d}B^{X(s)}_s}\mathbf{1}_{\mathcal{A}_{t,k}}\Bigr]\,,
\end{equation}
and
\begin{equation}\label{definition of u-hat-t}
\widehat{u}(t):=\mathbb{E}^X\Bigl[e^{ \int_0^t \mathrm{d}B^{X(s)}_s}\mathbf{1}_{\mathcal{A}_{t}}\Bigr]\,.
\end{equation}
For any given $n \in \mathds{N}^{\geq 1}$ and $k \in \mathds{N}$:\\
\underline{For $H> 1/2$} we have
$$
\mathbb{E}\widehat{u}_k(n)= \mathbb{E}^X \Bigl[\mathbf{1}_{\mathcal{A}_{n,k}}\mathbb{E}e^{ \int_0^n \mathrm{d}B^{X(s)}_s}\Bigr]\leq \mathrm{P}(\mathcal{A}_{n,k})e^{\frac{1}{2}n^{2H}}
$$
As in this case $\mathfrak{N}_{n}=n^2$, by the Poisson tail probability bound \eqref{poisson tail probability bound} we have
$$
\mathrm{P}(\mathcal{A}_{n,k})\leq (\frac{e \kappa n}{ n^{k+2}})^{ n^{k+2}}.
$$
\underline{For $H\leq 1/2$} we have
$$
\begin{aligned}
\mathbb{E}\widehat{u}_k(n)= \mathbb{E}^X \Bigl[\mathbf{1}_{\mathcal{A}_{n,k}}\mathbb{E}e^{ \int_0^n \mathrm{d}B^{X(s)}_s}\Bigr]
&\leq
\mathbb{E} \Bigl[\mathbf{1}_{\mathcal{A}_{n,k}}e^{ \frac{1}{2} J (\frac{n}{J})^{2H}}\Bigr]\\
&\leq
\mathrm{P}(\mathcal{A}_{n,k})e^{\frac{1}{2} \mathfrak{N}_n n^{k+1} (\frac{n}{\mathfrak{N}_n n^{k+1}})^{2H}}\,,
\end{aligned}
$$
where $J$ is the number of jumps of the random walk on $[0,n]$.\\
For this case $\mathfrak{N}_{n}=\lfloor\rho\kappa n\rfloor$, hence applying the Poisson tail probability bound \eqref{poisson tail probability bound} we have
$$
\mathrm{P}(\mathcal{A}_{n,k})\leq (\frac{e \kappa n}{ \rho \kappa n^{k+1}})^{ \rho \kappa  n^{k+1}}.
$$
\underline{So in both cases}, for $n$ large enough and every $k\in \mathds{N}$ we have
$$
\mathbb{E}\widehat{u}_k(n)\leq e^{-2n^{k+2}}.
$$
So by Markov's inequality, for $n$ large enough and every $k\in \mathds{N}$ we have
$$
\mathrm{P}\Bigl(\widehat{u}_k(n) \geq e^{-n^{k+2}}e^{-(k+1)(d+1)\log n}\Bigr)\leq n^{-(k+2)}\,,
$$
and hence
$$
\mathrm{P}\Bigl(\bigcup_{k\in \mathds{N}}\{\widehat{u}_k(n) \geq e^{-n^{k+2}}e^{-(k+1)(d+1)\log n}\}\Bigr)\leq 2n^{-2}\,.
$$
As the right hand side of this inequality is summable, Borel-Cantelli lemma implies that almost surely there exists $N_3$ such that for any $n\geq N_3$ and for any $k\in \mathds{N}$ we have
\begin{equation}\label{upper bound for u-hat-k-n}
\widehat{u}_k(n) \leq e^{-n^{k+2}}e^{-(k+1)(d+1)\log n}\,.
\end{equation}
\textbf{Step 3: }
Let $t\in \mathds{R}^{>0}$ be a given number with $t\geq \max\{N_1,N_2,N_3\}$. Define again $n:=\lceil t\rceil$.

For any $t_1, t_2\in \mathds{R}^{>0}$, let $\mathcal{C}_{t_1,t_2}$ be the event that the random walk has no jump in the time interval $[t_1, t_2]$. Using Equation \eqref{lower bound for Bn-Bt for x in Nn*n^k}, for any $k\in \mathds{N}$, we have
$$
\begin{aligned}
\widehat{u}_k(n)
&\geq \mathbb{E}^X\Bigl[e^{ \int_0^n \mathrm{d}B^{X(s)}_s}\mathbf{1}_{\mathcal{A}_{t,k}}\mathbf{1}_{\mathcal{C}_{t,n}}\Bigr]\\
&\geq e^{\inf_{|x|\leq \mathfrak{N}_n n^{k+1}}\inf_{n-1\leq t \leq n}(B_{n}^x-B_{t}^x)}\mathbb{E}^X\Bigl[e^{ \int_0^t \mathrm{d}B^{X(s)}_s}\mathbf{1}_{\mathcal{A}_{t,k}}\mathbf{1}_{\mathcal{C}_{t,n}}\Bigr]\\
&\geq e^{-(k+3)(d+1)\log n}\mathbb{E}^X\Bigl[e^{ \int_0^t \mathrm{d}B^{X(s)}_s}\mathbf{1}_{\mathcal{A}_{t,k}}\mathbf{1}_{\mathcal{C}_{t,n}}\Bigr]\\
&=e^{-(k+3)(d+1)\log n}\,\mathrm{P}(\mathcal{C}_{t,n})\, \widehat{u}_k(t)\,.
\end{aligned}
$$
Hence, using Equation \eqref{upper bound for u-hat-k-n}, we obtain the following inequality for any $k\in \mathds{N}$
\begin{equation}\label{upper bound for u-hat-k-t}
\begin{aligned}
\widehat{u}_k(t)
&\leq e^{\kappa} e^{(k+3)(d+1)\log n} \,\widehat{u}_k(n)
&\leq e^{-n^{k+3}}e^{\kappa} \leq e^{\kappa}e^{-n^{2}(k+1)}\,.
\end{aligned}
\end{equation}
In a similar way, using Equation \eqref{lower bound for Bt-B(n-1) for x in Nn} we get
\begin{equation}\label{Inequality between u-hat-t and u-hat-n}
\widehat{u}(t) \leq e^{\kappa} e^{\log n} \,\widehat{u}(n)=n e^{\kappa} \,\widehat{u}(n)\,,
\end{equation}
and
\begin{equation}\label{Inequality between u-hat-t and u-hat-(n-1)}
\widehat{u}(t) \geq e^{-\kappa}e^{-\log n} \,\widehat{u}(n-1)\,,
\end{equation}
which are valid for any $k\in \mathds{N}$.

So using Definitions \eqref{definition of u-hat-k-t} and \eqref{definition of u-hat-t}, and applying Inequalities \eqref{Inequality between u-hat-t and u-hat-(n-1)}, \eqref{upper bound for u-hat-k-t}, and \eqref{Inequality between u-hat-t and u-hat-n} we have
$$
u(t)=\widehat{u}(t)+\sum_{k=0}^{\infty} \widehat{u}_k(t)\leq n e^{\kappa} \,\widehat{u}(n)+e^{\kappa}\sum_{k=0}^{\infty}e^{-n^{2}(k+1)}\leq n e^{\kappa}  \,\widehat{u}(n)+e^{\kappa}e^{-n^{2}}
$$
So applying this inequality along with Equation \eqref{Inequality between u-hat-t and u-hat-(n-1)}, and noting the inequality $\log(\alpha+1)\leq \alpha$, we get
$$
\log \widehat{u}(n-1)-\delta_n \leq \log u(t) \leq \log \widehat{u}(n)+\Delta_n\,,
$$
where
$$
\Delta_n:
=\kappa+\log n+\frac{1}{n\,\widehat{u}(n) e^{n^{2}}} \;\quad \text{and} \;\quad \delta_n:=\kappa+\log n\,.
$$
In the next section, in Theorem \ref{Theorem of LowerBound}, we show that $\{\frac{\log \widehat{u}(n)}{n}\}_n$ converges to some strictly positive number (possibly $+\infty$ for $H>1/2$). Hence $\frac{\Delta_n}{f(n)}$ converges to zero as $n \rightarrow +\infty$. The convergence of $\frac{\delta_n}{f(n)}$ to zero as $n \rightarrow +\infty$ is also trivial. This completes the proof.
\end{proof}
\section{Lower Bound}\label{Lower Bound}
In this section we prove the positivity of $\lambda=\lim \frac{\widehat{U}(n)}{n}$ for any $H \in (0,1)$ and $\kappa>0$.

\begin{thm}\label{Theorem of LowerBound}
$\lambda=\lim_{n\rightarrow \infty} \frac{\widehat{U}(n)}{n}$ is strictly positive for every $H \in (0,1)$ and $\kappa >0$.
\end{thm}

The following well-known property of simple random walk on $\mathds{Z}$ plays a vital role in our argument (for a proof see e.g. \cite{Grinstead}).

\begin{lem}[First return to the origin]\label{first retuern to origin}
Let $\{S_n\}_n$ be a discrete-time random walk on $\mathds{Z}$ starting off the origin, i.e. $S_n=\sum_{k=1}^{n}X_k$ where $X_i \in \{-1,+1\}$ and $S_0=0$. Let $\nu_{2m}$ be number of different paths for the random walk to visit the origin for the first time at time $2m$, i.e. $S_{2m}=0$ but $S_{k}\neq 0$ for every $k \in \{1, \cdots, 2m-1\}$. The we have
$$
\nu_{2m}=\frac{1}{2m-1}\binom{2m}{m}
$$
\end{lem}

\begin{proof}[Proof of Theorem \ref{Theorem of LowerBound}]
For the $d$-dimensional simple random walk $X(\cdot)$ on $\mathds{Z}^d$, Let $\pi_i$ be the projection to the $i$-th coordinate; In other words if $X=(x_i)_i$, then for each $i$ we have $x_i:=\pi_i \mathrm{o} X$.

Let $ T:= 2 m d/{\kappa}$ with $m \in \mathds{N}^{\geq 1}$. For any $k\in \mathds{N}$, let $\mathcal{B}_k$ be the event that the random walk $X(\cdot)$ has the following property: for every $i \in \{1, \cdots, d\}$, the $i$-th projection, i.e. $\pi_i \mathrm{o} X$ be zero at time $k T$, make exactly $2 m$ jumps in the time interval $\bigl( k T,(k+1) T\bigr)$ and at its $2m$-th jump returns to zero for the first time. It is clear that then for each $i$, $\pi_i \mathrm{o} X$ doesn't change sign in the time interval $\bigl(k T,(k+1)T\bigr)$.
We have
$$
\frac{\widehat{U}(nT)}{nT}
\geq \,\frac{1}{nT}\,\mathbb{E} \log \mathbb{E}^X \Bigl(e^{\int_0^{nT}\mathrm{d}B^{X(s)}_s}\prod_{k=0}^{n-1}\mathbf{1}_{\mathcal{B}_k}\Bigr)\,.
$$
Using Markov property inductively, we have
$$
\mathbb{E} \log \mathbb{E}^X \Bigl(e^{\int_0^{nT}\mathrm{d}B^{X(s)}_s}\prod_{k=0}^{n-1}\mathbf{1}_{\mathcal{B}_k}\Bigr)
=\sum_{k=0}^{n-1}\mathbb{E} \log \mathbb{E}^X \Bigl(e^{\int_{kT}^{(k+1)T}\mathrm{d}B^{X(s)}_s}\mathbf{1}_{\mathcal{B}_k}\,\Bigl|\, X(kT)=0\Bigr)\,.
$$
So we have
$$
\begin{aligned}
\frac{\widehat{U}(nT)}{nT}
&\geq \frac{1}{nT}\sum_{k=0}^{n-1}\mathbb{E} \log \mathbb{E}^X \Bigl(e^{\int_{kT}^{(k+1)T}\mathrm{d}B^{X(s)}_s}\mathbf{1}_{\mathcal{B}_k}\,\Bigl|\, X(kT)=0\Bigr)\\
&= \frac{1}{T}\mathbb{E} \log \mathbb{E}^X \Bigl(e^{\int_{0}^{T}\mathrm{d}B^{X(s)}_s}\mathbf{1}_{\mathcal{B}_0}\Bigr)\,
\end{aligned}
$$
where we have used the time invariance of the random walk and the random environment, i.e. the fBm's.\\
Taking the limit when $n$ goes to $\infty$ we obtain
$$
\lambda \geq \frac{1}{T}\mathbb{E} \log \mathbb{E}^X \Bigl(e^{\int_{0}^{T}\mathrm{d}B^{X(s)}_s}\mathbf{1}_{\mathcal{B}_0}\Bigr).
$$
So it suffices to show the positivity of the right-hand-side of this inequality.

Let $\mathfrak{D}$ be the set of all possible paths of a $2md$-step discrete-time random walk on $\mathds{Z}^d$ started at the origin with the property that its projection over each coordinate makes exactly $2m$ jumps the last of which (the $2m$'th jump) is a first return to the zero site of that coordinate.
As $\mathcal{B}_0$ is an event that concerns only the number and direction of jumps of the random walk, NOT its jump times, conditional on the number of jumps it is independent of the jump times. Let $\mathbb{E}^{\mathbf{t}}$ denote the expectation with respect to the jump times when the number of jumps is $2md$, i.e. expectation with respect to the jump times $t_1, \cdots, t_{2md}$ distributed uniformly on a $2md$-dimensional simplex (in other words, $t_1, \cdots, t_{2md}$ is the ascending list of ${2md}$ uniformly distributed points on $(0,T)$). Let also $p_{m}$ be the probability that a simple random walk has $2md$ jumps in the time interval $[0,T]$.\\
We have
$$
\mathbb{E}^X \Bigl(e^{\int_{0}^{T}\mathrm{d}B^{X(s)}_s}\mathbf{1}_{\mathcal{B}_0}\Bigr)
=p_m \frac{1}{(2d)^{2md}} \sum_{j\in \mathfrak{D}} \mathbb{E}^{\mathbf{t}} \Bigl(e^{\int_{0}^{T}\mathrm{d}B^{X_j(s)}_s}\Bigr)\,.
$$
where $X_j$ is a continuous-time random walk $\mathds{Z}^d$ whose skeleton (i.e. the sequence of the sites it visits) is the same as $j\in \mathfrak{D}$. For each path $j$ in $\mathfrak{D}$ it is evident that $-j \in \mathfrak{D}$. So let $\mathfrak{D}/2$ be a subset of $\mathfrak{D}$ with the property that from each pair $(j, -j)$ contains only one; In other words it is the equivalence class of $\mathfrak{D}$ under the relation $j\sim i \; \Longleftrightarrow \; j=\pm i$. Then we have
$$
\mathbb{E}^X \Bigl(e^{\int_{0}^{T}\mathrm{d}B^{X(s)}_s}\mathbf{1}_{\mathcal{B}_0}\Bigr)
=p_m \frac{1}{(2d)^{2md}} \sum_{j\in \mathfrak{D}/2} \mathbb{E}^{\mathbf{t}} \Bigl(e^{\int_{0}^{T}\mathrm{d}B^{X_j(s)}_s}+e^{\int_{0}^{T}\mathrm{d}B^{-X_{j}(s)}_s}\Bigr)\,,
$$
hence
$$
\begin{aligned}
&\mathbb{E} \log \mathbb{E}^X \Bigl(e^{\int_{0}^{T}\mathrm{d}B^{X(s)}_s}\mathbf{1}_{\mathcal{B}_0}\Bigr)\\
&\qquad \qquad =\log p_m +\mathbb{E} \log \frac{1}{(2d)^{2md}} \sum_{j\in \mathfrak{D}/2} \mathbb{E}^{\mathbf{t}} \Bigl(e^{\int_{0}^{T}\mathrm{d}B^{X_j(s)}_s}+e^{\int_{0}^{T}\mathrm{d}B^{-X_{j}(s)}_s}\Bigr)\\
&\qquad \qquad \geq \log p_m + \frac{2}{|\mathfrak{D}|} \sum_{j\in \mathfrak{D}/2}\mathbb{E}^{\mathbf{t}}\mathbb{E}\log \frac{|\mathfrak{D}|}{(2d)^{2md+1}} \Bigl(e^{\int_{0}^{T}\mathrm{d}B^{X_j(s)}_s}+e^{\int_{0}^{T}\mathrm{d}B^{-X_{j}(s)}_s}\Bigr)\\
&\qquad \qquad = \log p_m + \log \frac{|\mathfrak{D}|}{(2d)^{2md+1}}  + \frac{2}{|\mathfrak{D}|} \sum_{j\in \mathfrak{D}/2}\mathbb{E}^{\mathbf{t}}\mathbb{E}\log \Bigl(e^{\int_{0}^{T}\mathrm{d}B^{X_j(s)}_s}+e^{\int_{0}^{T}\mathrm{d}B^{-X_{j}(s)}_s}\Bigr)\,.
\end{aligned}
$$
If $Y_1:=\int_{t_1}^{t_{2md}}\mathrm{d}B^{X_{j}(s)}_s$  and $Y_2:=\int_{t_1}^{t_{2md}}\mathrm{d}B^{-X_{j}(s)}_s$ we have
$$
\begin{aligned}
e^{\int_{0}^{T}\mathrm{d}B^{X_j(s)}_s}+e^{\int_{0}^{T}\mathrm{d}B^{-X_{j}(s)}_s}
&=e^{\int_{0}^{t_1}\mathrm{d}B^{X_j(s)}_s+\int_{t_{2md}}^{T}\mathrm{d}B^{X_j(s)}_s}
(e^{Y_1}+e^{Y_2})\\
&\geq e^{\int_{0}^{t_1}\mathrm{d}B^{X_j(s)}_s+\int_{t_{2md}}^{T}\mathrm{d}B^{X_j(s)}_s}
e^{\max\{Y_1,Y_2\}}.
\end{aligned}
$$
As $Y_1$ and $Y_2$ are independent identically-distributed zero-mean normal random variables we have
$$
\mathbb{E} \max\{Y_1,Y_2\} =\mathbb{E}\bigl( \frac{|Y_1-Y_2|+Y_1+Y_2}{2}\bigr)=\frac{\sigma}{\sqrt{\pi}}\,,
$$
where $\sigma^2$ is the variance of $Y_1$.
So we have
$$
\mathbb{E}^{\mathbf{t}}\mathbb{E}\log \Bigl(e^{\int_{0}^{T}\mathrm{d}B^{X_j(s)}_s}+e^{\int_{0}^{T}\mathrm{d}B^{-X_{j}(s)}_s}\Bigr)
\geq \mathbb{E}^{\mathbf{t}}(\sigma/\sqrt{\pi})\,.
$$
Let $\Delta:=t_1+(T-t_{2md})$, i.e. the total amount of time that the random walk spends at the origin during the time interval $[0,T]$. As $t_1, \cdots\, t_{2md}$ are uniformly distributed in $(0,T)$, it is clear that $\mathbb{E}^{\mathbf{t}}(\Delta)=2\frac{T}{2md+1}$.\\
$\bullet$ When $H\leq 1/2$, as the increments are negatively correlated (property \ref{positive or negative correlation}), staying in a single site gives a lower bound on the variance, i.e. $\sigma^2 \geq (T-\Delta)^{2H}$. Using the inequality $\alpha^{H} \geq \bigl(\frac{\alpha}{T}\bigr)\,T^{H}$ which holds for any $0\leq \alpha \leq T$ and $0<H<1$
we have $\sigma \geq \bigl(\frac{T-\Delta}{T}\bigr)\,T^{H}$. Hence
$$
\mathbb{E}^{\mathbf{t}}(\sigma)
\geq \mathbb{E}^{\mathbf{t}}\bigl(\frac{T-\Delta}{T}\bigr)\,T^{H} =
\frac{2md-1}{2md+1} \,T^{H} \asymp m^{H}\,.
$$
$\bullet$ When $H> 1/2$, as the increments are positively correlated (property \ref{positive or negative correlation}), visiting every site no more than once gives a lower bound on the variance, i.e. $\sigma^2 \geq \sum_{i=2}^{2md} (t_{i}-t_{i-1})^{2H}$. Also note that  the function $x^{2H}$ is convex for $H> 1/2$ and hence for any $\alpha_i \geq 0$, $i=1, \cdots, N$ we have
$$
\sum_{i=1}^{N} \alpha_i^{2H} \geq N\bigl(\frac{1}{N}\sum_{i=1}^{N} \alpha_i\bigr)^{2H}.
$$
So we have
$$
\begin{aligned}
\mathbb{E}^{\mathbf{t}}(\sigma)
&\geq \mathbb{E}^{\mathbf{t}}\sqrt{\sum_{i=2}^{2md}(t_{i}-t_{i-1})^{2H}}\\
&\geq \mathbb{E}^{\mathbf{t}}\sqrt{(2md-1)^{1-2H}\bigl(\sum_{i=2}^{2md}(t_{i}-t_{i-1})\bigr)^{2H}}\\
&= (2md-1)^{1/2-H}\,\mathbb{E}^{\mathbf{t}}\bigl(\sum_{i=2}^{2md}(t_{i}-t_{i-1})\bigr)^{H}\\
&\geq (2md-1)^{1/2-H}\,\mathbb{E}^{\mathbf{t}}\bigl(\frac{\sum_{i=2}^{2md}(t_{i}-t_{i-1})}{T}\bigr)\,T^{H}\\
&\geq (2md-1)^{1/2-H}\,\bigl(\frac{2md-1}{2md+1}\bigr)\,T^{H}\asymp \sqrt{m}\,.
\end{aligned}
$$
where we have used once again the inequality $\alpha^{H} \geq \bigl(\frac{\alpha}{T}\bigr)\,T^{H}$.

Hence we showed that
$$
\mathbb{E} \log \mathbb{E}^X \Bigl(e^{\int_{0}^{T}\mathrm{d}B^{X(s)}_s}\mathbf{1}_{\mathcal{B}_0}\Bigr)
\geq \log p_m + \log \frac{|\mathfrak{D}|}{(2d)^{2md+1}}  + C m^{\gamma} \,,
$$
where $C$ is some positive constant and $\gamma:=1/2$ for $H>1/2$ and $\gamma:=H$ for $H\leq1/2$.

$p_m$ is the probability that a Poisson random variable of mean $\kappa T= 2md$ has $2md$ jumps. So by Stirling's formula \cite{Robbins} we have
$$
p_m=e^{-2md}\frac{(2md)^{2md}}{(2md)!}\geq \frac{1}{2e\sqrt{\pi m d}}\,
$$
hence
$$
\log p_m \asymp -\log m\,.
$$

For determining $|\mathfrak{D}|$, first we notice that there are $\binom{2md}{2m \,\cdots \,2m}=\frac{(2md)!}{(2m)!^d}$ different ways of distributing the $2md$ jumps uniformly between the $d$ coordinates. For each $i=1, \cdots, d$, there are $\nu_{2m}$ different possible excursions for $\pi_i \mathrm{o} X$ such that it starts from zero, makes $2d$ jumps and at its $2d$-th jump returns to zero for the first time. So we have
$$
|\mathfrak{D}|=\frac{(2md)!}{(2m)!^d} \nu_{2m}^{d}=\frac{(2md)!}{(2m)!^d}\frac{(2m)!^d}{(m)!^{2d}}\frac{1}{(2m-1)^d}
=\frac{(2md)!}{(m)!^{2d}}\frac{1}{(2m-1)^d}.
$$
Again, by Stirling's formula we have
$$
\frac{(2md)!}{(m)!^{2d}} \asymp\frac{(2d)^{2md}}{m^{2d-1/2}},
$$
and hence
$$
\log \frac{|\mathfrak{D}|}{(2d)^{2md+1}} \asymp -\log m\,.
$$

This shows that
$$
\mathbb{E} \log \mathbb{E}^X \Bigl(e^{\int_{0}^{T}\mathrm{d}B^{X(s)}_s}\mathbf{1}_{\mathcal{B}_0}\Bigr)
\geq -C_1 \log m + C m^{\gamma} \,,
$$
which guarantees the positivity of this expression for $m$ large enough and hence completing the proof.
\end{proof}
\section{Upper Bound}\label{Upper Bounds}
In this section we establish an upper  bound on $\widehat{U}(T)$. For $H\leq 1/2$, we obtain an upper bound that is linear in $T$, which shows that $\lambda$ is finite.
For $H\geq 1/2$ the problem is much more complicated and we have only been able to prove that $\widehat{U}(T)$ and hence $U(T)$ grow at most like $T\sqrt{\log(T)}$.


\begin{thm}
For $H\leq 1/2$, the limit $\lim_{T\rightarrow\infty} \frac{\widehat{U}(T)}{T}=\lambda$ is finite.
\end{thm}
\begin{proof}
By convexity of $\log$ and using Jensen's inequality and then by the negative correlation of the fBms' increments (property \ref{positive or negative correlation}), we have
$$
\begin{aligned}
\widehat{U}(T)
&\leq \log \mathbb{E}^X \Bigl(\mathbb{E}\,e^{\int_0^T \mathrm{d}B_s^{X(s)}}\mathbf{1}_{\mathcal{A}_T}\Bigr) \\
&= \log \mathbb{E}^X \Bigl(\, e^{\frac{1}{2} \,var( \int_0^T \mathrm{d}B_s^{X(s)})} \mathbf{1}_{\mathcal{A}_T}\Bigr)\\
&\leq \log \mathbb{E}^X \Bigl(\, e^{\frac{1}{2} \,\sum_{i=0}^{n}(t_{i+1}-t_i)^{2H}}\mathbf{1}_{\mathcal{A}_T}\Bigr)\,,
\end{aligned}
$$
where $\{t_i\}_{i}$ are the jump times of the random walk $X(\cdot)$ in $(0,T)$, including the end points, and $n$ is the number of jumps. Then as the function $x^{2H}$ is concave, by Jensen's inequality we have
$$
\frac{1}{n+1} \sum_i (\triangle t_i)^{2H} \leq \Bigl(\frac{\sum_i \triangle t_i}{n+1}\Bigr)^{2H}=\Bigl(\frac{T}{n+1}\Bigr)^{2H}.
$$
But under the event $\mathcal{A}_T$, the number of jumps is smaller than $ \mathfrak{N}_T=\rho T$. So
$$
\begin{aligned}
\widehat{U}(T)
&\leq \log \mathbb{E}^X \Bigl(\, e^{\frac{1}{2}(n+1)^{1-2H}T^{2H}}\mathbf{1}_{\mathcal{A}_T}\Bigr)\\
&\leq \log \mathbb{E}^X \Bigl(\, e^{\frac{1}{2}(\rho T+1)}\mathbf{1}_{\mathcal{A}_T}\Bigr)\\
&\leq \frac{1}{2}(\rho T+1).
\end{aligned}
$$
This shows that $\lambda =\lim_{T\rightarrow\infty} \frac{\widehat{U}(T)}{T}$ is finite.
\end{proof}

When $H>1/2$, we apply a more elaborate method.
\begin{thm}\label{theorem on the upperbound for U}
For $H> 1/2$, we have $\widehat{U}(n)\curlyeqprec n \sqrt{\log n}$.
\end{thm}
\begin{proof}
We chop up the interval $[0,n]$ into subintervals $\{[l,l+1]\}_{l=0}^{n-1}$ and decompose each integral $\int_l^{l+1} \mathrm{d}B_s^{X(s)}$ into two parts: the residue part that comes from the Brownian motions contributions up to time $l-1$, and the innovation part that comes from the Brownian motions contributions from the interval $[l-1, l+1]$. We expect the innovation part to be the main contribution to the integral, and the residue part to be reasonably small.

We begin by the Volterra representation \eqref{Volterra representation} of a fBm. For $l\in \mathds{N}^{\geq 2}$ and $l\leq t_1 < t_2 \leq l+1$, we have
\begin{equation}\label{Fractional Increment Decomposition}
B_{t_2}-B_{t_1}=\int_0^{l-1}\Bigl(K_H(t_2,s)-K_H(t_1,s)\Bigr)\mathrm{d}W_s + Z_{t_2}-Z_{t_1}\,,
\end{equation}
where
\begin{equation}\label{Definition of Z}
Z_t:=\int_{l-1}^{t} K_H(t,s)\mathrm{d}W_s\,.
\end{equation}
For $0 \leq t \leq 2$  we also define $Z_t$ by
\begin{equation}\label{Definition of Z for t smaller than 2}
Z_t:=\int_{0}^{t} K_H(t,s)\mathrm{d}W_s\,.
\end{equation}
Applying the stochastic Fubini theorem \cite{Protter} to the first right-hand-side term of \eqref{Fractional Increment Decomposition} we have
\begin{equation}\label{Integral of Y}
\begin{aligned}
\int_0^{l-1}\Bigl(K_H(t_2,s)-K_H(t_1,s)\Bigr)\mathrm{d}W_s &=c_H\int_0^{l-1}\int_{t_1}^{t_2} (u-s)^{H-\frac{3}{2}}(\frac{u}{s})^{H-\frac{1}{2}} \, \mathrm{d}u\; \mathrm{d}W_s\\
&=\int_{t_1}^{t_2}  Y_l(u)\; \mathrm{d}u\,,
\end{aligned}
\end{equation}
where
\begin{equation}\label{definition of Y}
Y_l(u):=c_H\int_0^{l-1} (u-s)^{H-\frac{3}{2}}(\frac{u}{s})^{H-\frac{1}{2}} \, \mathrm{d}W_s\,.
\end{equation}

Applying this procedure to the family $\{B^x\}_{x\in \mathds{Z}^d}$, there exists a family of independent standard Brownian motions $\{W^x\}_{x\in \mathds{Z}^d}$ such that
$$
B^x_t=\int_0^t K_H(t,s) \, \mathrm{d}W^x_s\,.
$$
So for each site $x \in \mathds{Z}^d$, the processes $Y_l^x$ and $Z^x$ can be defined as above.

Back to the integral $\int_0^{n} \mathrm{d}B_s^{X(s)}$, it can be easily verified that
\begin{equation}\label{Integral of B decomposed to Z and Y}
\int_0^n \mathrm{d}B^{X(t)}_t=\int_0^{n} \mathrm{d}Z^{X(t)}_t + \sum_{l=2}^{n-1}\int_l^{l+1} Y_l^{X(t)}(t) \mathrm{d}t\,.
\end{equation}
We will show that the first term grows linearly in $n$ and the second term grows no faster than $n\sqrt{\log n}$.

For $\int_0^{n} \mathrm{d}Z^{X(t)}_t$, the idea is that by adjoining some terms to it we may turn it into a summation of mostly independent terms and hence getting a linear upper bound.
Indeed, let $\{\widetilde{W}^{l,x}\}_{x \in \mathds{Z}^d ,\, l\in \mathds{N}}$ be a family of independent standard Brownian motions, independent of any process introduced so far, in particular independent of the random walk $X(.)$, the fractional Brownian motions $\{B^x\}_{x \in \mathds{Z}^d}$ and hence their corresponding Brownian motions $\{W^x\}_{x \in \mathds{Z}^d}$ appearing in their integral representation. For any $l\in \mathds{N}^{\geq 2}$ and $x \in \mathds{Z}^d$ define $\widehat{W}^{l,x}$ as
$$
\widehat{W}^{l,x}:=
\begin{cases}
\widetilde{W}^{l,x}_s  &\text{for } \qquad s\in[0,l-1]\\
W^x_s- W^x_{l-1} + \widetilde{W}^{l,x}_{l-1}   &\text{for }\qquad s\in(l-1, \infty)\,.
\end{cases}
$$
and for $l=0, 1$, define $\widehat{W}^{l,x}:=W^x$.

It is easily verified that $\widehat{W}^{l,x}$ is itself a standard Brownian motion and hence the following expression
\begin{equation}\label{Definition of artificial B}
\widehat{B}^{l,x}_t:= \int_0^t K_H(t,s) \mathrm{d}\widehat{W}^{l,x}_s=
\int_0^{l-1} K_H(t,s)\mathrm{d}\widetilde{W}^{l,x}_s
+ \int_{l-1}^t K_H(t,s)\mathrm{d}W^x_s\,,
\end{equation}
is a fractional Brownian motion of Hurst parameter $H$.\\
Note also that for any $x \in\mathds{Z}^d$ and $l \leq t < l+1$, we have
$$
Z^x_t=\int_{l-1}^t K_H(t,s)\mathrm{d}W^x_s\,.
$$

By the same procedure as in equations \eqref{Fractional Increment Decomposition} through \eqref{definition of Y}, for any $t \in [l,l+1)$ we have
$$
\int_l^{l+1} \mathrm{d}Z^{X(t)}_t=\int_l^{l+1} \mathrm{d}\widehat{B}^{l, X(t)}_t -\int_l^{l+1}\widehat{Y}_l^{X(t)}(t) \mathrm{d}t\,,
$$
where
$$
\widehat{Y}_l^{x}(t):=c_H\int_0^{l-1} (u-s)^{H-\frac{3}{2}}(\frac{u}{s})^{H-\frac{1}{2}} \, \mathrm{d}\widetilde{W}^{l,x}_s \quad \text{for}\quad t \in [l,l+1)\,.
$$
We therefore have
$$
\int_0^{n} \mathrm{d}Z^{X(t)}_t=\sum_{l=0}^{n-1}\int_l^{l+1}\mathrm{d}\widehat{B}^{l, X(t)}_t-\sum_{l=2}^{n-1}\int_l^{l+1}\widehat{Y}_l^{X(t)}(t) \mathrm{d}t\,.
$$
This along with \eqref{Integral of B decomposed to Z and Y} implies that
\begin{equation}\label{final decomposition of the integral with respect to fractionals}
\int_0^{n}\mathrm{d} B^{X(t)}_t=
\sum_{l=0}^{n-1}\int_l^{l+1}\mathrm{d}\widehat{B}^{l, X(t)}_t-\sum_{l=2}^{n-1}\int_l^{l+1}\widehat{Y}_l^{X(t)}(t) \mathrm{d}t+\sum_{l=2}^{n-1}\int_l^{l+1}Y_l^{X(t)}(t) \mathrm{d}t\,.
\end{equation}
So we have
\begin{equation}\label{upper bound on U hat for H larger than half}
\begin{aligned}
\widehat{U}(n)
&=\mathbb{E} \log \mathbb{E}^X \Bigl(e^{\int_0^{n}\mathrm{d} B^{X(t)}_t}\mathbf{1}_{\mathcal{A}_n}\Bigr)\\
&\leq \mathbb{E} \log \mathbb{E}^X e^{\sum_{l=0}^{n-1}\int_l^{l+1}\mathrm{d}\widehat{B}^{l, X(t)}_t}+\sum_{l=2}^{n-1} \mathbb{E} \Bigl(\sup_{\substack{|x|\leq n^2\\ \;l \leq u \leq l+1}}|\widehat{Y}_l^{x}(u)|+\sup_{\substack{|x|\leq n^2\\ \;l \leq u \leq l+1}}|Y_l^{x}(u)|\Bigr)
\end{aligned}
\end{equation}\,.

First we find an upper bound on the first right-hand-side term.
Here we need an easy observation. Let $\widetilde{\sigma}^{l}$ be the sigma field generated by $\{\widetilde{W}^{l,x}_s \;;\; s \in [0,l-1] \,,\, x \in \mathds{Z}^d\}$ and $\sigma^{l}$ be the sigma field generated by $\{{W^{x}_s- W^{x}_{l-1}}\;;\; s \in (l-1,l+1]\,,\, x \in \mathds{Z}^d\}$. It is evident by \eqref{Definition of artificial B} that for any $l \leq t < l+1$ the process $\widehat{B}^{l,x}_t$ is measurable with respect to $\sigma^{l} \vee \widetilde{\sigma}^{l}$ where $\vee$ denotes the smallest sigma field containing the both. So $\int_l^{l+1} \mathrm{d}\widehat{B}^{l, X(t)}_t$ is also measurable with respect to $\sigma^{l}\vee \widetilde{\sigma}^{l}$.  As $\sigma^{l} \vee \widetilde{\sigma}^{l}$ and $\sigma^{k} \vee \widetilde{\sigma}^{k}$ are independent for $|k-l|\geq 2$, this shows that $\int_l^{l+1} \mathrm{d}\widehat{B}^{l, X(t)}_t$ and $\int_k^{k+1} \mathrm{d}\widehat{B}^{l, X(t)}_t$ are independent for $|k-l|\geq 2$. Hence, using the inequality $\mathbb{E}XY\leq \frac{1}{2}(\mathbb{E}X^2+\mathbb{E}Y^2)$, we have
$$
\var \sum_{l=0}^{n-1} \int_l^{l+1} \mathrm{d}\widehat{B}^{l, X(t)}_t  \leq 3\sum_{l=0}^{n-1} \var \int_l^{l+1} \mathrm{d}\widehat{B}^{l, X(t)}_t\,.
$$
We also notice that
\begin{equation}\label{varaiance of the integral over unit interval is bounded by one}
\var \int_l^{l+1} \mathrm{d}\widehat{B}^{l, X(t)}_t \leq 1\,,
\end{equation}
which follows from the fact that the upper bound is attained when the random walk stays in a single site for the whole time interval $[l,l+1)$ as the increments of the fBm are positively correlated (property \ref{positive or negative correlation}).

Hence we have
\begin{equation}\label{linear upper bound for the first term of the integral wrt fractionals}
\begin{aligned}
\mathbb{E}\Bigl(e^{\sum_{l=0}^{n-1}\int_l^{l+1} \mathrm{d}\widehat{B}^{l, X(t)}_t} \Bigr)
& =e^{\frac{1}{2} \var \sum_{l=0}^{n-1} \int_l^{l+1}\mathrm{d}\widehat{B}^{l, X(t)}_t}\\
&\leq e^{\frac{3}{2} \sum_{l=0}^{n-1} \var \int_l^{l+1}\mathrm{d}\widehat{B}^{l, X(t)}_t}\\
&\leq e^{\frac{3}{2}n}\,,
\end{aligned}
\end{equation}

Now we turn to the second right-hand-side term of Equation \eqref{upper bound on U hat for H larger than half}.\\
Applying Dudley's theorem, for any $l \in \mathds{N}^{\geq 2}$ we have
$$
\mathbb{E}\Bigl(\sup_{\substack{|x|\leq n^2\\ \;l \leq u \leq l+1}}|Y_l^{x}(u)|\,\Bigr)
\leq
K\int_0^\infty \sqrt{\log N(\varepsilon)}\, \mathrm{d}\varepsilon\,,
$$
where $K$ is a universal constant.\\
Using proposition \ref{Holder continuity of Y for H smaller than half}, for any $u, v \in [l,l+1]$ we have
$$
\mathbb{E}\bigl[Y_{l}(u)-Y_{l}(v)\bigr]^2 \curlyeqprec (u-v)^2.
$$
Particularly the upper bound doesn't depend on $l$.\\
So with the same argument given in Section \ref{Super-additivity}, it follows that there are positive numbers $M_1$ and $M_2$ depending only on $H$, such that $N(\varepsilon)\curlyeqprec \frac{1}{\varepsilon}$ for $0 < \varepsilon \leq M_1$, $N(\varepsilon)\asymp n^{2d}$ for $M_1 \leq \varepsilon < M_2$ and finally $N(\varepsilon)=1$ for $\varepsilon > M_2$ and. So there exists a positive constant $K_1$ such that for every $l$
$$
\mathbb{E}\Bigl(\sup_{\substack{|x|\leq n^2\\ \;l \leq u \leq l+1}}|Y_l^{x}(u)|\,\Bigr)
\leq K_1 \sqrt{\log n}.
$$
The same is true for $\widehat{Y}_l^{x}$
$$
\mathbb{E}\Bigl(\sup_{\substack{|x|\leq n^2\\ \;l \leq u \leq l+1}}|\widehat{Y}_l^{x}(u)|\,\Bigr)
\leq K_2 \sqrt{\log n}.
$$

Hence we have
$$
\widehat{U}(n)\leq 3/2n+K n\sqrt{\log n},
$$
where $K$ is a positive constant that doesn't depend on anything other than $H$.
\end{proof}
\section{Compact-Space Setup}\label{Compact Case}
In this section, we consider the compact-space model studied in \cite{Viens}. It turns out that our method to obtain an upper bound on $U(t)$ for the case of independent sites of $\mathds{Z}^d$, can be modified to give a much stronger upper bound in the compact set-up. Indeed we show that in the compact setup, $U(t)$ grows linearly in $t$ and hence $u(t)$ grows exponentially in $t$. This is in contrast with \cite{Viens} where its authors tried to show that $U(t)$ grows at least as fast as $\frac{t^{2H}}{\log t}$.
We identified the passage from (41) to (42) in Section 6.2 of \cite{Viens} as a probable source of the discrepancy with our article, which, when combined with other delicate arguments in \cite{Viens}, lead them to obtain an incorrect lower bound when $H>1/2$. In particular, the passage from (41) to (42) in \cite{Viens}, which is detailed at the bottom of the page where those equations appear therein, seems to be justified by invoking spatial homogeneity of the potential $W$, when in reality the authors of \cite{Viens} should have investigated the distribution of the potential $W$ conditional on the past and current values of the maximized path. In particular, if the maximized path's increment from time step $k$ to time step $j$ in their argument happens to be large, then for a $W$ which has a strong decay of spatial correlation, a modification of their argument can probably work. But we believe that when this increment is small, the argument is incorrect, and leads to a bias in the quantitative estimates of the lower bound later in \cite{Viens}. The fact that this discrepancy occurs for $H>1/2$ means that this bias is likely to manifest itself in a positive way because of the positivity of the increments of fractional Brownian motion in that case; this is indeed what appears to happen,  as the lower bound in Theorem 6.7 in \cite{Viens} is larger than it should be.

Replacing $\mathds{Z}^d$ by $\mathcal{S}^1$ (unit circle) is equivalent to considering the model on $\mathds{R}$ with a $2\pi$-periodic covariance function, i.e. a positive semi-definite function $Q:\mathds{R}\times\mathds{R}\rightarrow \mathds{R}$ such that
$$
Q(x+2\pi, y)=Q(x,y)=Q(x, y+2\pi)\quad \text{for every }x,y\in \mathds{R}\,.
$$
We additionally assume that $Q(x,y)$ has a positive order H\"older continuity. In other words there exist positive constants $C$ and $\alpha$ such that
\begin{equation}\label{Holder continuity condition of the covariance}
|Q(x,y)-\frac{1}{2}Q(x,x)-\frac{1}{2}Q(y,y)|\leq C |x-y|^\alpha \quad \text{for every }x,y\in \mathds{R}\,.
\end{equation}
Now we consider a family of fractional Brownian motions $\{B^x\}_{x \in \mathds{R}}$ with the following space covariance structure
$$
\mathbb{E}\bigl(B_t^x B_s^y\bigr)=R_H(t,s) \, Q(x,y)\,,
$$
where $R_H(t,s)$ is as in \eqref{fractional covariance function}, i.e.
$$
R_H(t,s):=\frac{1}{2}(|t|^{2H}+|s|^{2H}-|t-s|^{2H})\,.
$$
Defining $\{W^x\}_{x\in \mathds{R}}$ as a family of standard Brownian motions with space covariance structure $Q(\cdot,\cdot)$, i.e. $\mathbb{E}(W_t^{x}W_s^{y})=\min(s,t)\, Q(x,y)$, we can easily verify that all the arguments in the proof of theorem \ref{theorem on the upperbound for U} up to Equation \eqref{Integral of B decomposed to Z and Y} hold true in this new setting as well. Then we define $\{\widetilde{W}^{l,x}\}_{x \in \mathds{Z}^d ,\, l\in \mathds{N}}$ as a family of standard Brownian motions independent of all the other random processes involved, with the following (space) covariance structure:
$$
\mathbb{E}\bigl(\widetilde{W}_t^{l,x}\widetilde{W}_s^{k,y}\bigr)=\min(s,t)\, \delta_k(l) \, Q(x,y)\,.
$$
So in particular $\widetilde{W}^{l,x}$'s are independent for different $l$'s, but are correlated for different $x$'s.

We can easily show that with this definition everything goes well until Equation \eqref{upper bound on U hat for H larger than half}. There we have
\begin{multline}\label{upper bound on U for H larger than half in the compact setting}
U(n)=\mathbb{E} \log \mathbb{E}^{X} \Bigl(e^{\int_0^{n}\mathrm{d} B^{X(t)}_t}\Bigr) \leq  \mathbb{E} \log \mathbb{E}^X e^{\sum_{l=0}^{n-1}\int_l^{l+1}\mathrm{d}\widehat{B}^{l, X(t)}_t}\\
+\sum_{l=2}^{n-1} \mathbb{E} \Bigl(\sup_{\substack{x \in \mathds{R}\\ \;l \leq u \leq l+1}}|\widehat{Y}_l^{x}(u)|+\sup_{\substack{x \in \mathds{R}\\ \;l \leq u \leq l+1}}|Y_l^{x}(u)|\Bigr)\,.
\end{multline}

We can easily verify that the arguments following Equation \eqref{upper bound on U hat for H larger than half} up to Equation \eqref{varaiance of the integral over unit interval is bounded by one} are still valid. Using the fact that the increments of a fBm with a Hurst parameter larger than half are positively correlated (property \ref{positive or negative correlation}) we have
$$
\var \int_l^{l+1} \mathrm{d}\widehat{B}^{l, X(t)}_t \leq M\,,
$$
where $M:=\sup_{x,y\in \mathds{R}}Q(x,y)$. Note that the finiteness of $M$ is guaranteed by Condition \eqref{Holder continuity condition of the covariance}. So we have
$$
\mathbb{E} e^{\sum_{l=0}^{n-1}\int_l^{l+1}\mathrm{d}\widehat{B}^{l, X(t)}_t}\leq e^{\frac{3}{2}M n}\,,
$$
which establishes a linear upper bound on the first term of the right-hand-side of Equation \eqref{upper bound on U for H larger than half in the compact setting}.

Turning to the second right-hand-side term of Equation \eqref{upper bound on U for H larger than half in the compact setting}, we show that basically the same arguments provide us with a linear upper bound instead of $K\,n\sqrt{\log n}$ which was the best upper bound we managed to get in the non-compact setting.\\
Indeed, let $W$ be a standard Brownian motion independent of any process already defined and let $Y_l(\cdot)$ be related to $W$ as in Equation \eqref{definition of Y}. Using the polarization identity we have $[{W}^{l,x}, {W}^{l,y}]_t=t\,Q(x,y)$ where $[{W}^{l,x}, {W}^{l,y}]_t$ denotes the covariation of the two Brownian motions \cite{Protter}. So using It\={o}'s isometry \cite{Protter} we have
$$
\mathbb{E}\bigl( Y_l^{x}(u) Y_l^{y}(u)\bigr)=Q(x,y)\,\mathbb{E}\bigl( Y_l(u) Y_l(u)\bigr)\,.
$$
Applying proposition \ref{Holder continuity of Y for H smaller than half} and Condition \eqref{Holder continuity condition of the covariance}, for any $u, v \in [l,l+1]$ we have
$$
\mathbb{E}\bigl(Y_l^{x}(u)-Y_{l}^y(v)\bigr)^2 \leq C \bigl(|u-v|^2+|x-y|^\alpha \bigr)\,,
$$
where $C$ only depends on $H$. In particular it does not depend on $l$, $u$, $v$, $x$, or $y$.

Let $\{x_i\}_{i=1}^m$ be the $m$ equally-distanced points in the interval $(-\pi, \pi)$ (or equivalently on $\mathcal{S}^1$), and $\{t_j\}_{j=1}^p$ be the $p$ equally-distanced points in the interval $(l, l+1)$. Then with the notation of \ref{Dudley theorem}, the set of $\varepsilon$-balls centered at $\{Y^{x_i}_{l}(t_j)\}_{i,j}$ cover the whole of $\{Y^{x}_{l}(t)\}_{\substack{x\in \mathds{R}\\ \;l \leq t \leq l+1}}$ provided that $\sqrt{C\bigl((\frac{\pi}{m})^\alpha+(\frac{1}{2p})^2\bigr)}<\varepsilon$. So the corresponding $N(\varepsilon)$ in Dudley's theorem is bounded by $C_1 \varepsilon^{-\gamma}$ for every $\varepsilon\leq \varepsilon_0$ where $C_1$, $\gamma$, and $\varepsilon_0$ are strictly positive constants which only depend on $H$ and $\alpha$. So by Dudley's theorem, there exists a positive constant $K'_1$ which depends only on $H$ and $\alpha$ such that for every $l$ we have
$$
\mathbb{E}\Bigl(\sup_{\substack{x\in \mathds{R}\\ \;l \leq u \leq l+1}}|Y_l^{x}(u)|\,\Bigr)
\leq K'_1\,.
$$
One can easily verify that the same is true for $\widehat{Y}_l^{x}$:
$$
\mathbb{E}\Bigl(\sup_{\substack{x\in \mathds{R}\\ \;l \leq u \leq l+1}}|\widehat{Y}_l^{x}(u)|\,\Bigr)
\leq K'_2.
$$
This establishes an $n$-linear upper bound on $U(n)$ as claimed.

\bibliographystyle{plain}


\def\polhk#1{\setbox0=\hbox{#1}{\ooalign{\hidewidth
  \lower1.5ex\hbox{`}\hidewidth\crcr\unhbox0}}} \def\cprime{$'$}

\end{document}